\documentclass[a4paper,english,oneside,12pt]{amsart}
\usepackage[latin1]{inputenc}
\usepackage[normalem]{ulem}

\usepackage{amsmath}
\usepackage{amssymb}
\usepackage{amsfonts}
\usepackage{latexsym}
\usepackage{mathtools}
\usepackage{amscd}
\usepackage{xy} \xyoption{all}
\usepackage{hyperref}
\usepackage{enumerate}
\usepackage{color}

\newcommand{\GFE}{%\annotation{GFE}
\cite{bce:gfe}}
\newcommand{\GFEx}[1]{%\annotation{GFE}
\cite[#1]{bce:gfe}}

\newcommand{\ntoone}[1]{\mbox{$#1$--to--1}}

\newcommand{\inva}{\mathfrak I}

\newcommand{\symN}{\mathcal S_N}    

\usepackage[latin1]{inputenc}
\usepackage{amsmath}
\usepackage{amssymb}
\usepackage{amsfonts}
\usepackage{latexsym}
\usepackage{mathtools}
\usepackage{amscd}
\usepackage{xy} \xyoption{all}
\usepackage{hyperref}
\usepackage{enumerate}

\newcommand{\pib}{\overline{\pi}}
\newcommand{\Bb}{\overline{B}}
\newcommand{\Xb}{\overline{X}}
\newcommand{\gb}{\overline{G}}
\newcommand{\vb}{\overline{\mathcal V}}
\newcommand{\Vb}{\overline{\mathcal V}}

\newcommand{\pit}{\widetilde{\pi}}

\newcommand{\symk}{\mathcal S_k}

\newcommand{\symE}{\mathcal S_E}
\newcommand{\symG}{\mathcal S_G}

\newcommand{\TOKEok}[1]{}%\marginpar{$\surd$}}
\newcommand{\rac}{\gamma}

\newcommand{\FE}[1]{\left[\left[#1\right]\right]_{FE}}

\newcommand{\mulset}{\mathsf{M}^{-1}}
\newcommand{\mulsetdown}{\mathsf{M}}
\newcommand{\rM}{|_{\mathsf{M}^{-1}}}

\newcommand{\mugr}[1]{G[#1]}

\newcommand{\BFgroup}[1]{\mathsf{BF}(#1)}

\newcommand{\cok}{\mathsf{cok}}
\newcommand{\N}{\mathbb{N}}
\newcommand{\Z}{\mathbb{Z}}
\newcommand{\R}{\mathbb{R}}
\newcommand{\T}{\mathbb{T}}

\newcommand{\lan}[1]{\mathsf{L}(#1)}

\newcommand{\al}{{\mathfrak{a}}}
\newcommand{\alp}[1]{\al(#1)}
\newcommand{\maptorus}[1]{\mathsf{S}{#1}}
\newcommand{\mt}[1]{\mathsf{S}{#1}}
\newcommand{\expansion}[2]{\mathfrak{e}_{{#1}}(#2)}
\newcommand{\cover}{\mathfrak{c}}
\newcommand{\labe}{\mathcal{L}}
\newcommand{\osh}[1][]{\sigma_{#1}}
\newcommand{\follower}[1]{\mathsf{F}(#1)}

\renewcommand{\phi}{\varphi}

\newtheorem{lemma}{Lemma}[section]
\newtheorem{corollary}[lemma]{Corollary}
\newtheorem{theorem}[lemma]{Theorem}
\newtheorem{prop}[lemma]{Proposition}
\newtheorem{proposition}[lemma]{Proposition}
 
\theoremstyle{definition}
\newtheorem{defi}[lemma]{Definition}
\newtheorem{definition}[lemma]{Definition}

\newtheorem{fact}[lemma]{Fact}
\newtheorem{facts}[lemma]{Facts}
\newtheorem{example}[lemma]{Example}
\newtheorem{examples}[lemma]{Examples}
\newtheorem{remark}[lemma]{Remark}

\input xy
\xyoption{all}
\numberwithin{equation}{section}

\title{Flow equivalence of sofic shifts}
\author{Mike Boyle}
\author{Toke Meier Carlsen}
\author{S\o ren Eilers}
\date{\today}

\begin{document} 
\begin{abstract}
We classify certain sofic shifts (the irreducible
Point Extension Type, or PET, sofic shifts) up to flow equivalence, 
using invariants of the canonical Fischer cover.  
There are two main ingredients. \begin{enumerate}
\item An extension theorem,  for extending flow equivalences of 
subshifts to flow equivalent irreducible shifts of finite type which contain them. 
\item The classification of certain constant to one maps from SFTs 
via algebraic invariants of associated $G$-SFTs.  \end{enumerate}
\end{abstract}

\maketitle

\tableofcontents

\section{Introduction}

Suppose for $i=1,2$ that $\sigma_i\colon X_i \to X_i$ is a 
homeomorphism of a compact metric space. Then 
$\sigma_1$ and $\sigma_2$ are {\it flow equivalent} if they 
are topologically conjugate to return maps to cross 
sections of a common flow. Equivalently, there 
is a homeomorphism between their mapping tori 
which respects orbits and orientation of the 
suspension flows.  
Return maps to cross sections of flows 
arose historically in the study of 
differential equations, and were abstracted to 
topological dynamics on compact 
metric spaces \cite{ns:book}; they were later  
absorbed into a theory of 
cocycles for general group actions 
\cite{hk:book}.  

For zero-dimensional systems, 
flow equivalence is a multifaceted relation.
 Flow equivalence of shifts of finite type turns out to be a significant 
feature of an overall algebraic framework 
for analyzing the topological dynamics of SFTs 
and free $G$-SFTs (SFTs with a continuous shift commuting 
free action by a finite group $G$) 
(see \cite{BSullivan} and its references). 
Flow equivalence is a significant 
ingredient in the ongoing interplay between 
symbolic dynamics and $C^*$-algebras
(see \cite{Matsumoto2013} and its many references).
Especially, the classification of shifts of finite type 
up to flow equivalence 
(by Franks \cite{jf:fesft} in the irreducible case 
and by Huang \cite{mb:fesftpf, mbdh:pbeim, dh:fersft}
in general)
was a critical 
ingredient in the classification of 
Cuntz-Krieger algebras of real rank zero up to stable 
isomorphism \cite{restorff,rordam}. 
Flow equivalence of some subshift 
classes has 
been studied with category theory,   
semigroup theory and sophisticated 
formal language analysis 
\cite{costasteinberg,kriegerfe}. 
The Barge-Diamond topological classification 
\cite{mbbd:citosts} of  one-dimensional 
substitution tiling spaces amounts to the classification 
of the associated substitution subshifts up to 
flow equivalence, and a recent paper by Johansen (\cite{rj:xxx}) addresses the case of beta-shifts.

In this paper, we study  flow equivalence 
of sofic shifts  (the class of subshifts 
most closely generalizing the shifts of finite type), 
and especially the 
almost finite type (AFT) class of Marcus \cite{bm:ssed},
by means of their canonical 
covers. There is a natural notion of 
flow equivalence of such covers, which turns 
out to be a complete invariant of flow equivalence 
of the associated sofic shifts. We use the Fischer covers  
covers to classify a large (but far from general) class 
of AFT shifts up to flow equivalence.  There are 
two main ingredients for this.

 First, relying on \cite{mbwk:ass} we prove 
an extension theorem which lets us extend a 
flow equivalence of a subflow to a flow 
equivalence of mapping tori of flow equivalent irreducible 
SFTs.  With the known classification of 
irreducible SFTs up to flow equivalence, 
this reduces the flow equivalence  classification of 
AFT shifts to the problem of classifying  
the restrictions of their Fischer covers to 
their multiplicity shifts. 

{Next, we introduce a new class of sofic shifts,
  contained in the 
  class of   AFT shifts:  
  the  \lq\lq point extension type\rq\rq\ (PET) sofic shifts.
  A sofic shift $Y$ is PET if for its Fischer cover $\pi$ and
  each $k>1$, 
  the set $\mulsetdown_k(\pi ):=\{y\in Y\colon |\pi^{-1}(y)| =k\}$ is closed.
  For these shifts, drawing further on
work of Adler, Kitchens and Marcus \cite{akmfactor}
(following Rudolph \cite{rudolphcounts}),}
 we reduce the analysis of these restricted covers  
to the problem of 
classifying 
associated G-SFTs
up to equivariant flow equivalence. 
We solve that problem 
in  a separate paper \GFE. The 
complete algebraic invariants are manageable 
for constructions and for analyzing some 
classes. 
We do not know if there is a 
decision procedure 
for determining $G$-flow equivalence in 
general, but many cases can be decided 
(see  Remark \ref{fullforce} and \cite{bs:decide}). 
 
There is 
a natural notion of shift equivalence for sofic shifts
\cite{mbwk:amsess,HamachiNasu} 
which by work of Kim and Roush is decidable 
\cite{KimRoushsoficse}.
The algebraic framework for that work should also be appropriate for 
studying flow equivalence of sofic shifts and 
its decidability. 
Ultimately, {the} flow equivalence of sofic shifts 
{may} be understood from both viewpoints, each 
with its advantages.

%gps \cite{gps} 

The paper has the following structure. 
Sections \ref{sec:fe} and \ref{sec:ms}
are background. Sections \ref{sec:et} and \ref{sec:rt} 
have our main result (the Extension Theorem) and its main 
application (the Reduction Theorem). The Reduction Theorem 
is applied in Section \ref{sec:rt} 
 to classify the 
near Markov shifts up to 
flow equivalence and establish that such shifts spaces are always flow
equivalent to their time-reversals. In Section 
\ref{sec:np}, 
following Adler, Kitchens and Marcus \cite{akmfactor,akmgroup} 
we explain how topological conjugacy 
of certain constant to one maps is equivalent to the topological 
conjugacy 
of certain associated $G$-SFTs. In Section 
\ref{sec:petalg}, 
we apply this  to classify  PET (point extension 
type) sofic shifts up
to flow equivalence.
{
Finally, in Section
\ref{sec:petalg},
we give  procedures to determine
whether a sofic shift is PET and
to compute matrices over $\Z_+G$ from which
the algebraic invariants are defined. 
}

This work was supported by the Danish National Research Foundation through the Centre for Symmetry and Deformation (DNRF92), and by VILLUM FONDEN through the network for Experimental Mathematics in Number Theory, Operator Algebras, and Topology.

\section{Flow equivalence}\label{sec:fe} 
We give a detailed treatment of flow equivalence for 
subshifts in \cite{bce:fei}. Here we give just 
some basic properties needed for this paper. 
The reader is referred to \cite{dlbm:isdc} for further 
background on basic symbolic dynamics.

%\subsection{Basic properties}

For a finite alphabet $\al$ we will by $\osh[\al]$ (or just $\osh$ when the alphabet has been fixed) denote the shift
map on $\al^\Z$.
By a shift space $X$, we mean a closed shift invariant subset
  $X$ of  some $\al^\Z$; we use $X$ also to represent the 
dynamical system which is the restriction of the shift to $X$.
For a shift space $X$ we will by $\alp{X}$ denote the alphabet of $X$.

%{\bf Basic properties.}
For a shift space $X$, the \emph{mapping torus}
is the quotient space
\[
\maptorus{X}={X\times \R}/{\sim}
\]
with
\[
(x,t)\sim (y,s)\Longleftrightarrow t-s\in \Z\wedge y=\sigma^{t-s}(x).
\]
Note that $\maptorus{X}$ comes equipped with an action of $\R$:
\[
\rac_{t_0}([x,t])=[x,t+t_0].
\]

We say that two shift spaces $X$ and $Y$ are \emph{flow equivalent} if there exists a homeomorphism $\psi\colon\maptorus{X}\to \maptorus{Y}$ which is orientation preserving in the sense that for each $x\in X$ there is an increasing function $f_x\colon\R\to\R$ with the property
\[
\psi(\rac_t(x))=\rac_{f_x(t)}(\psi(x)).
\]
 We call the map $\psi$ 
an equivalence of flows, or a flow equivalence, and denote the flow equivalence class of $X$ by $\FE{X}$. A key notion for us is
that of a cross section:

\begin{defi}
A subset $R\subset \maptorus{X}$ is called a \emph{cross section} when the map $s\colon\R\times R\to \maptorus{X}$ given by
\[
s(t,x)=\rac_t(x)
\]
is a surjective local homeomorphism.
\end{defi}

If $R\subset \maptorus{X}$ is cross section, then there is a well defined \emph{return map} $\rho\colon R\to R$ such that $\rho(x)=\gamma_t(x)$ where $t=\min\{t>0\colon\gamma_t(x)\in R\}$. The dynamical system $(R,\rho)$ is conjugate to a shift space $X_R$, and $\mt X_R$ can in a natural way be identified with $\mt X$.
Note that $X\times \{0\}/\sim$ is always a cross section for $\maptorus X$ and that $(X\times \{0\}/\sim,\rho)$ is conjugate to $(X,\sigma)$.
Throughout this paper, we will identify $(X,\sigma)$ with $(X\times \{0\}/\sim,\rho)$.

\begin{defi}
When $F\subseteq \maptorus{X}$ is closed and invariant under $\rac_t$ we call $F$ a \emph{subflow}.
\end{defi}

For any sliding block code $f\colon Y\to X$ we define
$\maptorus{f}\colon\maptorus{Y}\to\maptorus {X}$ in the obvious 
way, $[(y,t)]\mapsto [(f(y),t)]$, and say that $f$ is flow
equivalent to $f'\colon Y'\to X'$ when there are flow equivalences $\phi,
\psi$ with
\[
\xymatrix{\maptorus{Y}\ar[r]^-{\psi}\ar[d]_-{\maptorus{f}}&\maptorus{Y'}\ar[d]^-{\maptorus{f'}}\\
\maptorus{X}\ar[r]_-{\phi}&\maptorus{X'}}
\]
In this case, suppressing the domains and codomains from the notation,
we write $\FE{f}=\FE{f'}$.

Of course, any conjugacy $\phi\colon Y\to X$ induces a flow equivalence
$\maptorus{\phi}\colon\maptorus{Y}\to\maptorus{X}$, by 
the rule $\gamma_t (y) \mapsto \gamma_t(\phi (y))$, for $y\in Y$ and 
$t\in \R$. 

We describe next the other basic move we use for flow
equivalence:  
{\it symbol expansion}.

%A flow self-equivalence $\chi\colon\maptorus{X}\to\maptorus{X}$
%which fixes all flow lines (i.e. for any $x$ and $s$ there is $t$ so
%that $\chi([x,t])=[x,s]$) we shall call an \textbf{inner isotopy}. 

%Let $G$ be a finite group acting on a shift space $X$ by $\alpha\colon X\times G\to X$, such that $\alpha(g)$ is a conjugacy for %every $g$.

%\subsection{Symbol expansion}
%{\bf Symbol expansion}
% \begin{lemma}
% Any irreducible SFT is flow equivalent to
% a mixing SFT.
% \end{lemma}
% \begin{proof}
% If the irreducible sofic shift $T$
% is not mixing, then for a Fischer cover of $T$ (left or right) 
% pass to higher block presentations in the 
% domain and range, until there are periodic points 
% $x$,$y$ in the range on disjoint symbols and $x$ has a 
% unique preimage $x'$. Let $y'$ be a periodic point in the 
% domain which is a preimage of $y$.  Now stretch a symbol 
% a used for $y$ into $a_1,\dots,a_k$, likewise stretching its 
% preimage symbols. There are periodic points $w'$,$z'$ in the 
% time-changed system corresponding to $x'$,$y'$. We have 
% $\period(y')=\period(z')$ and 
% $\period(w')=\period(x')+ k-1$. So we can arranged 
% relatively prime periods in the domain and hence 
% mixing. 
% \end{proof}

\begin{defi} \label{def:expansion}
  Let $\al$ be a finite alphabet and let $A\subseteq\al$. Choose for
  each $a\in A$ a symbol $\star_a$ which does not belong to
  $\al$. Let $\tilde{\al}:=\al\cup\{\star_a\colon
  a\in A\}$. For $a\in\al$ let 
  \begin{equation*}
    \tilde{a}:=
    \begin{cases}
      a&\text{if }a\notin A,\\
      a\star_a&\text{if }a\in A,
    \end{cases}
  \end{equation*}
  and let $\iota_A$ be the map from $\al^\Z$ to $\tilde{\al}^\Z$ which maps
  \begin{align*}
    \dots a_{-2}a_{-1}&.a_0a_1a_2\dots\shortintertext{to}
    \dots
    \tilde{a}_{-2}\tilde{a}_{-1}&.\tilde{a}_0\tilde{a}_1\tilde{a}_2\dots .
  \end{align*}
  We will also use $\iota_A$ to denote the map from $\al^*$ to
  $\tilde{\al}^*$ which maps a word $a_1a_2\dots a_n$ to
  $\tilde{a}_1\tilde{a}_2\dots\tilde{a}_n$.

  For a subshift $X$ of $\al^\Z$ we will by $\expansion{A}{X}$ denote
  the shift space $\{\iota_A(x)\colon x\in
  X\}\cup\{\osh[\tilde{\al}](\iota_A(x))\colon x\in X\}$.
%  If $a\in\al$, then we will
%  write $\iota_a$ instead of 
%  $\iota_{\{a\}}$, and $\expansion{a}{T}$ instead of $\expansion{\{a\}}{T}$. 
\end{defi}

\begin{lemma} \label{lemma:1}
  Let $X$ and $Y$ be shift spaces, let $\pi\colon Y\to X$ be a one-block
  factor map and let $\labe\colon\alp{Y}\to\alp{X}$ be the map that
  induces $\pi$. Let $A\subseteq\alp{X}$ and let
  $\tilde{A}:=\{b\in\alp{Y}\colon \labe(b)\in A\}$.  Choose for each $a\in
  A$ a symbol $\star_a$ which does not belong to $\alp{X}$, and choose for
  each $b\in\tilde{A}$ a symbol $\star_b$ which does not belong to
  $\alp{Y}$. Let $\tilde{\labe}$ be the map from $\alp{Y}\cup\{\star_b\colon
  b\in \tilde{A}\}$ to $\alp{X}\cup\{\star_a\colon a\in A\}$ that maps $c$
  to $\labe(c)$ for $c\in\alp{Y}$ and $\star_b$ to $\star_{\labe(b)}$ for
  $b\in\tilde{A}$, and let $\expansion{A}{\pi}$ be the restriction of
  the one-block factor map induced by $\tilde{\labe}$ to
  $\expansion{\tilde{A}}{Y}$. 

  Then $\expansion{A}{\pi}$ is a factor map from
  $\expansion{\tilde{A}}{Y}$ to $\expansion{A}{X}$, and the two factor
  maps $\pi$ and $\expansion{A}{\pi}$ are flow equivalent.
\end{lemma}

\begin{proof}
  It is straightforward to check that
  $\expansion{A}{\pi}(\expansion{\tilde{A}}{Y})=\expansion{A}{X}$.

  Let $\phi$ be the map from $\maptorus{X}$ to
  $\maptorus{\expansion{A}{X}}$ defined by 
  \begin{equation*}
    \phi([x,t])=
    \begin{cases}
      [\iota_A(x),2t]&\text{if $x_0\in A$ and $t\in [0,1/2[$},\\
      [\osh[\alp{T}\cup\{\star_a\colon a\in A\}](\iota_A(x)),2t-1]&\text{if
        $x_0\in A$ and $t\in [1/2,1[$},\\ 
      [\iota_A(x),t]&\text{if }x_0\notin A,
    \end{cases}
  \end{equation*}
  for $x\in X$ and $t\in [0,1[$, and let $\tilde{\phi}$ be the map from 
  $\maptorus{Y}$ to $\maptorus{\expansion{\tilde{A}}{Y}}$ defined by 
  \begin{equation*}
    \tilde{\phi}([y,t])=
    \begin{cases}
      [\iota_{\tilde{A}}(y),2t]&\text{if $y_0\in \tilde{A}$ and $t\in
        [0,1/2[$},\\ 
      [\osh[\alp{S}\cup\{\star_b\colon
      b\in\tilde{A}\}](\iota_{\tilde{A}}(y)),2t-1]&\text{if $y_0\in
          \tilde{A}$ and $t\in [1/2,1[$},\\ 
      [\iota_{\tilde{A}}(y),t]&\text{if }y_0\notin \tilde{A},
    \end{cases}
  \end{equation*}
for $y\in Y$ and $t\in [0,1[$. It is not difficult to check that
$\phi$ and $\tilde{\phi}$ are flow equivalences and that the diagram 
\begin{equation*}
    \xymatrix{Y\ar[r]^{\tilde{\phi}}\ar[d]_{\pi}&
      \expansion{\tilde{A}}{Y}\ar[d]^{\expansion{A}{\pi}}\\
      X\ar[r]^{\phi}&
      \expansion{A}{X}}
  \end{equation*}
  commutes. Thus $\pi$ and $\expansion{A}{\pi}$ are flow equivalent
  factor maps. 
\end{proof}

If $a\in\al$, then we will write $\expansion{a}{\pi}$ instead of
$\expansion{\{a\}}{\pi}$, and $\iota_a$ instead of $\iota_{\{a\}}$.

We will need to make frequent reference to the classical invariants 
of Parry and Sullivan \cite{parrysullivan} and Bowen and Franks
\cite{BowenFranks}, which were shown to 
be complete invariants of flow equivalence for 
infinite irreducible SFTs by  
Franks \cite{jf:fesft}.  
When a shift of finite type is given by an $n\times n$ adjacency
matrix $A$, the invariant consists of the Bowen-Franks group
$\cok(I-A)$ along with the sign of $\det(I-A)$.
A zero entropy irreducible SFT 
is a cyclic permutation of 
finitely many points, whose mapping torus is a circle; 
in this case the 
Bowen-Franks group is $\Z$,  which is also 
the Bowen-Franks group of 
some infinite SFTs.

\section{Multiplicity sets of canonical covers}\label{sec:ms}

%{\bf Canonical covers.}

We assume some familiarity with basic 
symbolic dynamics; \cite{dlbm:isdc} is an excellent basic 
reference. In this section we sketch some of the background.

\begin{defi}
  By a \emph{cover} we mean a pair of maps $(\cover,\pi_\cover)$
  defined on the class of irreducible sofic shifts which to an
  irreducible sofic shift $X$ associate an irreducible shift of finite
  type $\cover(X)$ and a factor map $\pi_{\cover(X)}\colon\cover(X)\to X$. 

%FUNCTORIAL???

  We will say that such a cover $(\cover,\pi_\cover)$ is
  \emph{canonical} if the following 
  holds: If $X_1$ and $X_2$ are irreducible sofic shifts and
  $\phi\colon X_1\to X_2$ is a conjugacy, then there exists a 
unique conjugacy
  $\cover(\phi)\colon \cover(X_1)\to\cover(X_2)$ such that the diagram 
  \begin{equation*}
    \xymatrix{
      \cover(X_1)\ar[r]^{\cover(\phi)}\ar[d]_{\pi_{\cover(X_1)}} 
      &\cover(X_2)\ar[d]^{\pi_{\cover(X_2)}}\\
      X_1\ar[r]^{\phi}&X_2}
  \end{equation*}
  commutes.

  We will say that a cover $(\cover,\pi_\cover)$ \emph{respects
    symbol expansion} if the following holds: The factor map
  $\pi_{\cover(X)}$ is a one-block code for all irreducible sofic
  shifts $X$, and if $X$ is an
  irreducible sofic shift, $a\in\alp{X}$,
  $\labe_{\cover(X)}\colon \alp{\cover(X)}\to\alp{X}$ is the map that
  induces $\pi_{\cover(X)}$, and $A=\{b\in\alp{\cover(X)}\colon 
  \labe_{\cover(X)}(b)=a\}$, then there exists a
  conjugacy $\phi\colon \cover(\expansion{a}{X})\to\expansion{A}{\cover(X)}$
  such that the diagram 
  \begin{equation*}
    \xymatrix{
      \cover(\expansion{a}{X})\ar[rr]^{\phi}
      \ar[rd]_{\pi_{\cover(\expansion{a}{X})}} 
      &&\expansion{A}{\cover(X)}\ar[ld]^{\expansion{a}{\pi_{\cover(X)}}}\\
      &\expansion{a}{X}&}
  \end{equation*}
  commutes.
\end{defi}

Krieger has in \cite{KriegerSoficI} proved that the right Fischer cover is
canonical. We will now record that it also respects symbol expansion.

\begin{prop} \label{prop:fischer}
  The right Fischer cover respects symbol expansion.
\end{prop}

\begin{proof}
  Let $X$ be an irreducible sofic shift. The right Fischer cover of
  $X$ can be constructed as follows. Say that a word $u\in\lan{X}$ (the language of $X$) is
  \emph{magic} (also known as \emph{intrinsically synchronizing}) if it has the
  following property: If $vu,uw\in\lan{X}$, then $vuw\in\lan{X}$. For
  a magic word $u$ let $\follower{u}:=\{v\in\lan{X}\colon 
  uv\in\lan{X}\}$, and let $(G,\labe)$ be the labeled graph with vertex set
  $\{\follower{u}\colon  u\in\lan{X}\text{ is magic}\}$ and where there for
  two magic words $u$ and $v$ and a symbol $a\in\alp{X}$ is an edge
  from $\follower{u}$ to $\follower{v}$ labeled $a$ if and only if
  $\follower{v}=\follower{ua}$. Then the right Fischer cover $(Y,\pi)$
  of $X$ is
  the edge shift $Y$ of $G$ together with the one-block code $\pi$ induced
  by $\labe$. 

  Let $a\in\alp{X}$. It is not difficult to check that we have
  \begin{multline*}
    \{u\in\lan{\expansion{a}{X}}\colon  u\text{ is magic}\}\\
    =\{\star_a\iota_a(u),\iota_a(u)a,\iota_a(u)\in\lan{\expansion{a}{X}}\colon 
    u\in\lan{X}\text{ is magic}\},
  \end{multline*}
  and that if $u\in\lan{X}$ is magic, then
  $$\follower{\iota_a(u)}=\{\iota_a(v),\iota_a(w)a\in\lan{\expansion{a}{X}}\colon 
  v,w\in\follower{u}\},$$ if $u\in\lan{X}$ is magic and
  $\iota_a(u)a\in\lan{\expansion{a}{X}}$, then
  $$\follower{\iota_a(u)a}
  =\{\star_a\iota_a(v),\star_a\iota_a(w)a\in\lan{\expansion{a}{X}}\colon 
  v,w\in\follower{ua}\},$$ and if $u\in\lan{X}$ is magic and
  $\star_a\iota_a(u)\in\lan{\expansion{a}{X}}$, then
  $\follower{\star_a\iota_a(u)}=\follower{\iota_a(au)}$.
  Thus if $(\hat{Y},\hat{\pi})$ is the right Fischer cover of
  $\expansion{a}{X}$ and $\hat{\labe}$ is the labeling which induces
  $\hat{\pi}$, and
  $A=\{e\in\alp{Y}\colon  \labe(e)=a\}$, then
  there exists a bijection $\eta$ from $\alp{\expansion{A}{Y}}$ to
  $\alp{\hat{Y}}$ which maps $e$ to the edge from $\follower{\iota_a(u)}$ to
  $\follower{\iota_a(v)}$ labeled $b$ if $e$ is the edge from $\follower{u}$ to
  $\follower{v}$ labeled $b$ and $b\ne a$, maps $e$ to the
  edge from $\follower{\iota_a(u)}$ to $\follower{\iota_a(u)a}$
  labeled $a$ if $e$ is the edge from $\follower{u}$ to
  $\follower{ua}$ labeled $a$, and maps $\star_e$ to the edge from
  $\follower{\iota_a(u)a}$ to $\follower{\iota_a(ua)}$ labeled $\star_a$
  if $e$ is the edge from $\follower{u}$ to $\follower{ua}$ labeled
  $a$. If $\tilde{\labe}$ is the map from $\alp{\expansion{A}{Y}}$ to
  $\alp{\expansion{a}{X}}$ which maps $e$
  to $\labe(e)$ for $e\in\alp{Y}$ and $\star_e$ to $\star_a$ for
  $e\in A$, then $\hat{\labe} \eta=\tilde{\labe}$, and it follows
  that if $\phi$ is the conjugacy from $\expansion{A}{Y}$ to
  $\hat{Y}$ induced by $\eta$, then the diagram 
  \begin{equation*}
    \xymatrix{
      \expansion{A}{Y}\ar[rr]^{\phi}
      \ar[rd]_{\expansion{a}{\pi}} 
      &&\hat{Y}\ar[ld]^{\hat{\pi}}\\
      &\expansion{a}{X}&}
  \end{equation*}
  commutes.
\end{proof}

One can in a similar way prove that the left Fischer cover, the right
and the left Krieger cover, the predecessor set cover and the follower
set cover all respect symbol expansion. It is shown in
\cite{KriegerSoficI} that the left Fischer cover, the right
and the left Krieger cover also are canonical.

We now have:

\begin{theorem} \label{theorem:prereduction}
  Let $(\cover,\pi_\cover)$ be a cover which is canonical and respects
  symbol expansion. 
  When $X_1$ and $X_2$ are flow equivalent irreducible sofic shifts
the two factor maps $\pi_{\cover(X_1)}$ and
    $\pi_{\cover(X_2)}$ are flow equivalent factor maps. In other words,
\[
\FE{X_1}=\FE{X_2}\Longrightarrow \FE{\pi_{\cover{X_1}}}=\FE{\pi_{\cover{X_2}}}.
\]
\end{theorem}
\begin{proof}
  Since flow equivalence among shift spaces is generated by conjugacy
  and symbol expansion, it is enough to prove that if $X_1$ and $X_2$ are
  conjugate, then $\pi_{\cover(X_1)}$ and $\pi_{\cover(X_2)}$ are
  conjugate factor maps, and 
  that if $a\in\lan{X_1}$ and $X_2=\expansion{a}{X_1}$, then $\pi_{\cover(X_1)}$
  and $\pi_{\cover(X_2)}$ are flow equivalent factor maps. The first of these
  assertions is exactly the assertion that $\cover$ is canonical, and the
  second follows from the assumption that $\cover$ respects symbol
  expansion and Lemma \ref{lemma:1}. 
\end{proof}

%{\bf Multiplicity.}
Absolutely central to our approach in this 
paper will be the restriction of covers to their 
multiplicity sets.

\begin{defi} \label{defn:multiplicity}
Given a finite to one map $\pi \colon Y\to X$, we 
define 
\begin{align*} 
\mathsf{MultiCard}(\pi ) \ &=\ 
\{ k\in \N\colon  k>1 \text{ and } \exists x\in X , |\pi^{-1}(x)|=k\}, \\ 
\mulsetdown_k (\pi)\ & =\ \{x\in X\colon
|\pi^{-1}(x)| = k \}, \\  
\mulset_k (\pi) \ & =
\ \pi^{-1}( \mulsetdown_k (\pi)), \\
\mulsetdown(\pi) \ & =\bigcup_{k>1}\mulsetdown_k (\pi), \\
\mulset(\pi) \ & =\bigcup_{k>1}\mulset_k (\pi). \\
\end{align*} 
\end{defi}
 
We denote these sets as \emph{multiplicity sets} and note that they are always
 shift invariant. Thus, they become a shift spaces in their own right precisely when they are closed. In this case, restricted factor maps such as
\[
\pi\rM\colon \mulset(\pi)\to \mulsetdown(\pi)
\]
are defined. We note 

\begin{corollary} \label{corollary:prereduction}
  Let $(\cover,\pi_\cover)$ be a cover which is canonical, respects
  symbol expansion and is right or left resolving. When $X_1$ and $X_2$ are flow equivalent irreducible sofic shifts, then
  \[
  \mathsf{MultiCard}({\pi_{\cover(X_1)}})=  \mathsf{MultiCard}({\pi_{\cover(X_2)}}).
  \]
  \end{corollary}

\begin{proof}
Because of Theorem \ref{theorem:prereduction} it suffices to prove that for such a cover $\pi\colon Y\to X$ we can infer whether $k\in \mathsf{MultiCard}(\pi)$ directly from the homeomorphism class of the map
\[
\maptorus\pi\colon \maptorus Y\to\maptorus X.
\]
For this, we note that 
since the covers are either right or left resolving, we see that $y\in Y$ is periodic precisely when $\pi(y)$ is (although possibly with another period).
Assume now that $k\in \mathsf{MultiCard}(\pi)$ and fix  $\pi(y_1)=\dots=\pi(y_k)=x$ with all $y_i\not= y_j$, $i\not=j$.

When $x$ is not periodic, neither is any of the $y_i$, and hence in $\maptorus X$ there is a flow line 
homeomorphic to $\R$ having precisely 
  $k$ distinct flow lines homeomorphic to $\R$ in its preimage. 
When $x$ is periodic, so are all $y_i$, and hence in $\maptorus Y$ there are $m\leq k$ distinct flow lines homeomorphic to $\T$ mapping to the same flow line homeomorphic to $\T$. The winding numbers for each of these maps will be positive and sum up to $k$.

In the other direction, whenever one of these configurations can be found, $x$ and $y_1,\dots, y_k$ can be constructed.
 \end{proof}

\begin{remark}
%We remark that when a sofic shift $T$ is 
%not AFT, the multiplicity of a factor map $\phi$ from an SFT 
%onto $T$ is not always minimized at a Fischer cover. The problem of even determining 
%the smallest possible $m(\phi)$ is open and difficult. 
%See \cite{MR1836227}. 
%
Note that for a general Fischer cover $\pi\colon  Y\to X$ (even with $X$ AFT), it can happen that neither 
$\mulset(\pi)$ nor $\mulsetdown(\pi)$ is SFT (\cite[pp. 60-61]{mbwk:amsess}).
\end{remark}

\section{The extension theorem}\label{sec:et}

The proof of our extension result 
for flow equivalences of subflows is based on the
following extension result for conjugacies of subshifts,
which is a direct consequence of \cite[Theorem 1.5]{mbwk:ass}. 

\begin{theorem}\cite{mbwk:ass} \label{frombk}
Let $X$ be a mixing SFT with disjoint subshifts $Y$ and $Y'$  such that
\begin{enumerate}
\item There is a conjugacy $\phi\colon Y\to Y'$.
\item For every positive integer $n$, $X\setminus Y$ contains at least two orbits of cardinality $n$.
\end{enumerate}
Then there is a self-conjugacy $\widetilde{\phi}\colon X\to X$ such that
$\widetilde{\phi}|_{Y}=\phi$. 
Moreover, $\widetilde{\phi}$ can be chosen to act trivially on the
dimension group of $X$.
\end{theorem}

We expect the \lq\lq Moreover\rq\rq \  statement of the following theorem
(Theorem \ref{extension-thm}),  regarding the induced isomorphism 
on the {\it isotopy futures group}, $\mathcal F(X)$, 
may be useful 
for future coding arguments,  but we do not use it for the flow
equivalence 
results in the current paper. So we postpone a short review of 
$\mathcal F(X)$  
to the end of this section.

\begin{theorem}\label{extension-thm} (Extension Theorem) 
Suppose $X$ and $X'$ are flow equivalent irreducible SFTs 
with proper subsystems $Y$,$Y'$ which are flow equivalent through 
\[
\phi\colon \maptorus{Y}\to 
\maptorus{Y'}.
\]
Then there is a flow equivalence
\[
\widetilde{\phi}\colon \maptorus{X}\to 
\maptorus{X'}
\]
which agrees with $\phi$ on $\maptorus{Y}$. 

Moreover, given an isomorphism $b\colon 
\mathcal F(X)\to \mathcal F(X')$, 
$\widetilde{\phi}$ can be chosen such that the isomorphism 
$\mathcal F(X)\to \mathcal F(X')$ induced by $\widetilde{\phi}$ equals 
 $b$. 
\end{theorem}

The condition of properness is necessary, as a proper subshift $X'\subset X$ may be flow equivalent to $X$ without the embedding extending.

\begin{proof}[Proof of Theorem \ref{extension-thm}] 
We will prove the theorem in several steps. Every step but the last is
a reduction to a special case. We give a complete proof of the
existence 
of $\widetilde{\phi}$ before discussing the action on 
the isotopy futures group. We will use Lemma \ref{makeroom} which we state and prove just after this proof.

{\it Step 1: Reduction to the case $X=X'$. } 

Let $\psi \colon  \mt X \to \mt X'$ be a flow equivalence. 
Suppose there is a flow equivalence $\kappa \colon  \maptorus X \to \maptorus X$ 
which is an extension of the flow equivalence 
$\psi^{-1}  \phi\colon  \mt Y\to \mt Y$.
Then  $\psi   \kappa \colon  \mt X \to \mt
X'$ is a flow equivalence extending 
$\phi\colon \mt Y \to \mt Y'$. 

So from here,  without loss of generality we assume $X=X'$. Also, 
set $Y_1=Y$ and $Y_2 = Y'$. 

{\it Step 2: Reduction to the case $X$ is mixing.} 

There is a flow equivalence $\gamma \colon  \mt X \to \mt X''$ where $X''$ is a 
mixing SFT (this is most easily seen by appealing to \cite{jf:fesft}). 
As in Step 1, 
if we can find a flow equivalence $\mt X''\to \mt X''$ extending 
$\gamma\phi\gamma^{-1}\colon  \mt (\gamma Y_1) \to \mt(\gamma Y_2)$, 
then we can pull it back to a flow equivalence 
$\mt X \to \mt X$ extending $\phi$. 

{\it Step 3: Reduction to the case 
$X$ contains a subshift $Y_3$ conjugate to the 3-shift, and 
the subshifts $Y_1,Y_2,Y_3$ of $X$  are 
pairwise disjoint.} 

Given  $N$ and $W= Y_1 \cup Y_2$ in $X$, 
we take $\gamma$, $g$ and $B$ 
as given by 
Lemma \ref{makeroom}.  
We choose  $N$ large enough that 
a disjoint union of the subshift 
$Y_2$ and two copies of the 3-shift 
embed in $X_N$ (the full $N$-shift). 
The subshift $g(Y_2)$ is conjugate to
$Y_2$, so there is an  
embedding $f$ of $g(Y_2)$ 
into $X_{B(22)}= X_N$ with image disjoint from a copy 
$Y_3$ of the 3-shift contained in $X_N$. 

By Theorem \ref{frombk}, 
there is  an automorphism  $k$ of $X_B$ which 
restricts to $f$ on $g(Y_2)$ and restricts to 
the identity on $Y_3$. (The extra copy of the 
3-shift provides enough periodic orbits.) 
Let $\kappa = \mt k\colon  \mt X_B \to \mt X_B$.  
Now $\mt (g(Y_1))$,  $\mt( \kappa g(Y_2) )$ and $\mt Y_3$ 
are pairwise disjoint in $\mt X_B$. 

Suppose we can find a flow equivalence 
$\alpha\colon  \mt X_B \to \mt X_B $ 
which extends the flow equivalence 
$\kappa \gamma \phi \gamma^{-1} \colon  \mt (g(Y_1)) \to \mt (\kappa g
(Y_2))$. 
Then $
\gamma^{-1}
\kappa^{-1} 
 \alpha \gamma $ 
 is a flow equivalence
$ \mt X \to \mt X $  which extends $\phi$. 
So, 
without loss of generality we
may assume that $X$ contains a subshift $Y_3$ conjugate 
to the 3-shift, with 
$Y_1,Y_2 $ and $Y_3$ pairwise disjoint.

{\it Step 4: Reduction to the case $\phi =\mt h\colon  Y_1\to Y_2$, with 
$h\colon  Y_1\to Y_2$ a topological conjugacy.} 

By the Parry-Sullivan argument in \cite{parrysullivan} 
(see \cite[Theorem 4.2]{bce:fei} for full details), there are closed sets 
$D_1\subseteq Y_1$ and $D_2 \subseteq Y_2$; a homeomorphism 
$h\colon  D_1\to D_2$; and a flow equivalence $\epsilon\colon  \mt Y_1 \to 
\mt Y_1$ isotopic to the identity such that the following hold: 
\begin{enumerate} 
\item For $k=1,2$: $D_k$ is a cross section of $\mt Y_k$
\item The restriction $(\phi\epsilon)|_{D_1}$ is a topological conjugacy of the return maps $\rho_i \colon  D_i \to D_i$
(i.e. $\phi\epsilon(D_1)=D_2$ and $(\phi\epsilon)|_{D_1} \rho_1 = \rho_2 (\phi\epsilon)|_{D_1}$).  
\end{enumerate} 

Let $k$ be 1 or 2. Then $D_k$ is a cross section for $\mt Y_k$ 
and $D_k\subseteq Y_k$. Therefore, $D_k$ is clopen in $Y_k$ 
and $Y_k$ is a discrete tower over $D_k$. For $x\in D_k$, 
define $\tau_k (x) = \min \{ j>0\colon  \sigma^j(x) \in D_k \}$, 
the first return time to $D_k$.
Set $T_k= \max \tau_k$ and 
$D_k(j)=\{ x\in D_k\colon  \tau_k (x) = j\}$.  
% and $D(i,j) = \sigma^i D(j)$.
Then 
\[
\{ \, \sigma^i (D_k(j))\colon  1\leq j \leq T_k, \, 0\leq i < j\} 
\]
is a clopen partition of $Y_k$ (some of the sets  $D_k(j)$ might be
empty). For $1\leq j \leq T_k$, 
choose $\widetilde{D}_k(j)$ a clopen subset  of $X$ 
such that 
$D_k(j) \subseteq \widetilde D_k (j)$ and 
$\{\sigma^i (\widetilde{D}_k(j))\colon  1\leq j \leq T_k, \, 0\leq i < j\}$
is a collection of 
 pairwise disjoint sets.  
Set $\widetilde D_k = \bigcup_{j=1}^{T_k} \widetilde D_k(j)$ and 
$\widetilde E_k = \bigcup_{j=1}^{T_k} \bigcup_{i=1}^{j-1} \sigma^i (\widetilde D_k
(j))$. 
Because the subshifts $Y_1, Y_2, Y_3$ are pairwise disjoint, 
we may also require 
\begin{equation*} 
\big( \widetilde{D}_1 \cup \widetilde{E}_1 \big) 
\cap 
\big( \widetilde{D}_2 \cup \widetilde{E}_2 \big)  = \emptyset,
\end{equation*}
and
\begin{equation*}
\Biggl( \bigcup_{k=1}^2 \bigl(\widetilde{D}_k \cup \widetilde{E}_k \bigr)\Biggr)  
\cap Y_3= \emptyset. 
\end{equation*} 
Define 
\begin{equation*} 
\widetilde{D}_k = 
X\setminus \bigcup_{k=1}^2 \bigl(\widetilde D_k \cup \widetilde E_k\bigr),
\end{equation*}
and
\begin{equation*}
\widetilde D =\bigcup_{k=1}^3\widetilde D_k. 
\end{equation*} 
Now $\widetilde D$ is a discrete cross section for $X$ 
(and a cross section for $\mt X$).  
Let $\widetilde{\rho} \colon  \widetilde D \to \widetilde D$ be the 
return map under the shift $\sigma$ (or equivalently, 
under the suspension flow 
on $\mt X$). 
If, for $k\in\{1,2\}$, $x\in D_k$, then $\widetilde{\rho}(x)=\rho_k(x)$, and if $x\in Y_3$, then $\widetilde{\rho}(x) = \sigma (x)$. 

Because $\widetilde D$ is a discrete cross section of the 
SFT $X$, there is a shift of finite type $X''$ 
and a flow equivalence $\mt X \to \mt X''$ such that $\gamma|_{\widetilde D}$ is a topological conjugacy from $(\widetilde D,\widetilde\rho)$ to $X''$.
For $k\in \{1,2\}$, set $Y''_k = \gamma (D_k)$, and set $Y''_3=\gamma(Y_3)$.
Because $Y_3 \subseteq \widetilde D$,  
the restriction $\gamma|_{Y_3} \colon  Y_3 \to Y_3''$ is a topological 
conjugacy of subshifts, and $Y_3''$ is conjugate to the 
3-shift. 

Let $\phi'' = \gamma \phi\epsilon \gamma^{-1} \colon 
\mt Y''_1 \to \mt Y''_2$ and $h''=\phi''|_{Y''_1}$. Because $\gamma|_{D_k}$ conjugates $(D_k,\rho_k)$ and $Y''_k$, and $\phi\epsilon$ conjugates $(D_1,\rho_1)$ and $(D_2,\rho_2)$, it follows that $h''$ 
defines a conjugacy of shifts 
$Y''_1 \to Y''_2$. There is therefore a flow equivalence $\epsilon''\colon  \mt Y''_1 \to 
\mt Y''_1$ isotopic to the identity such that $\phi''\epsilon''=\mt h''$.
% and therefore $\phi''$ is an isomorphism 
%of flows $\mt Y''_1 \to Y''_2$.  
If we can find $\kappa \colon  \mt X'' \to \mt X''$  a flow equivalence 
extending  $\phi''\epsilon''=\mt h''$, then we have that
$\gamma^{-1}\kappa \gamma \colon  \mt X \to \mt X$ is a 
flow equivalence extending $\phi\epsilon\gamma^{-1}\epsilon''\gamma$. Since $\epsilon\gamma^{-1}\epsilon''\gamma\colon \mt Y_1\to\mt Y_1$ is a flow equivalence which is isotopic to the identity, it follows from \cite[Proposition 7.1]{bce:fei} that there is a flow equivalence $\widetilde{ \epsilon}\colon  \mt X \to \mt X$ which is isotopic to the identity such that 
$\widetilde{\epsilon}$ equals $\epsilon\gamma^{-1}\epsilon''\gamma$ on $\mt Y_1$. Thus, $\gamma^{-1}\kappa \gamma\widetilde{\epsilon}^{-1}\colon  \mt X \to \mt X$ is a 
flow equivalence extending $\phi$.

So without loss of generality, in the next step we may assume  
there is a topological conjugacy of subshifts $h\colon Y_1 \to Y_2$ 
such that 
$\phi = \mt h\colon  Y_1 \to Y_2$ . 

{\it Step 5: Appeal to Extension Theorem for conjugacy.} 

We have  $\phi = \mt h \colon  \mt Y_1 \to \mt Y_2$ with 
$h\colon  Y_1 \to Y_2$ a conjugacy of subshifts of the 
mixing SFT $X$.  
For every $k$ in $\N$ the set 
$X\setminus Y_1 $ contains at least 
two $X$-orbits of cardinality $k$ 
(because this set contains $Y_3$, a copy of 
the 3-shift).  It follows from Theorem \ref{frombk} 
that 
$h$ extends to a conjugacy $k\colon X\to X$. Then the flow 
equivalence $\mt k \colon \mt X \to \mt X$ is an extension of 
$\phi = \mt h\colon  Y_1 \to Y_2 $.  This finishes the proof that 
the flow equivalence extending $\phi$ exists.

Now we turn to the ``Moreover'' claim. We assume
the background given at the end of this section. 
For a flow  equivalence $\beta$,  we use $[ \beta ] $ 
to denote the induced isomorphism of isotopy futures groups 
(an automorphism if $\beta $ is a self equivalence).  
We claim that the extension 
$\widetilde{\phi}$ produced in Steps 2-5 acts trivially on
$\mathcal F(X)$, 
 for the following 
reasons. 
\begin{enumerate} 
\item 
  The automorphisms $k$  of Steps 3 and 5, provided by
  Theorem \ref{frombk}, are chosen to be inert, so 
$[\mt k ] = \text{Id}$.  
\item 
If $[\beta ] = \text{Id}$, then 
$[\gamma \beta \gamma^{-1}] = [\gamma ][ \beta ][\gamma^{-1} ] =
\text{Id}$. 
\item 
For a flow equivalence $\widetilde{\epsilon} $ isotopic to the identity, 
$[\widetilde{\epsilon} ]  = \text{Id}$. 
\end{enumerate} 

Now suppose 
$ b \colon  \mathcal F(X) \to \mathcal F(X')$ is an isomorphism of  
isotopy futures groups and we 
want the flow equivalence $\mt X\to \mt X' $ 
extending  $\phi\colon  \mt Y \to \mt Y'$  to
induce the isomorphism $ b$. 
Let $\psi\colon \mt X\to \mt X'$  be a flow equivalence.
 Let $a\colon  \mathcal F(X)\to \mathcal F(X)$ be
an automorphism such that $b = [\psi ] a$.
By 
\cite[Theorem 7.13]{mb:fesftpf}, there is a flow equivalence 
$\alpha \colon  S_X\to S_X$ such that 
$[\alpha ] = a$.
Now $\alpha^{-1}\psi^{-1} \phi$ defines a flow equivalence
from $\mt Y $ to a submapping torus of $\mt X$. 
Apply the argument of Steps 2-5 to 
extend this to  a flow equivalence $\gamma \colon  \mt X \to \mt X$
such that $[\gamma ] = \text{Id}$ . 
Define  $\widetilde{\phi}=\psi \alpha \gamma$.
Then $\overline{\phi}$
is a flow equivalence $  \mt X \to \mt X' $ 
extending $\phi $ such that  
$[\widetilde{\phi} ]= [\psi][ \alpha ] = [\psi ] a =b$. 
This completes the proof.

\end{proof}

Now comes the lemma that we use in the proof of Theorem \ref{extension-thm}.

\begin{lemma}\label{makeroom} 
Suppose $X$ is a mixing SFT and $W$ is a proper subshift of $X$ and $N$ is a 
positive integer. Then there is a primitive matrix $B$ with 
$2\times 2$ block form $B =
\left( \begin{smallmatrix} B(11) & B(12) \\ B(21) & N 
\end{smallmatrix} \right)$  and a flow equivalence 
$\gamma \colon  \maptorus X \to \maptorus X_B$ and 
a topological conjugacy $g$ from $ W$ to a subshift
of $X_{B(11)}$ such that 
the restriction of $\gamma $ to $\mt W$ equals $\mt g$. 
\end{lemma} 
\begin{proof} 
Given integers $k\geq 2$ and $N$, let $Q_{k,N}$ be the $k\times
k$ matrix  $Q$ such that 
$
Q(i,i+1) = 1\ , \ 1\leq i < k$ ; 
$Q(k,1) = N$; and 
$Q(i,j) = 0 \  \text{otherwise}$. 
%\begin{align*} 
%Q(i,i+1) &= 1\ , \quad 1\leq i < k \\
%Q(k,1) &= N \\ 
%Q(i,j) &= 0 \ \ \quad \text{otherwise}. 
%\end{align*} 

Let $W'\neq X$ be a mixing SFT such that $W\subseteq W'\subset X$. Let 
$C$ be a primitive matrix such that $X_C$ is topologically conjugate
to $W'$.  
Given $N$, fix $k$ such that with 
$Q=Q_{k,N}$  and $E= 
\left( \begin{smallmatrix} C & 0 \\ 0 & Q 
\end{smallmatrix} \right)$, 
it follows from Krieger's Embedding Theorem that there 
is an embedding $\eta\colon  X_E \to X$ with $\eta (X_E)\neq X$.   
Using  a   modification of Krieger's Embedding 
Theorem proof \cite[Remark p.548]{B-LowerEntropyFactors}, 
we  require $\eta$ to be an
extension of the given 
conjugacy from $X_C$ to  $W'$. 

Next we let $\gamma$ be a conjugacy from $X$ to 
a higher block presentation $X_D$ of $X$, where $D$ is a primitive matrix with a principal 
submatrix $D(1)$ such that $\gamma$ maps $\eta (X_E)$ onto 
$X_{D(1)}$.  

Next we will appeal to Nasu's Masking Lemma which can be stated and proved in terms of
graphs (as in \cite[Lemma 3.18]{Nasumask} and \cite[Sec. 10.2]{dlbm:isdc})
or matrices (as in \cite[Appendix 1]{BHspectra}).
The matrix statement gives that if $M$ is a principal submatrix
of a square matrix $A$ over $\Z_+$, and a strong shift equivalence
over $\Z_+$ from $M$ to a matrix $M'$ is given, then it
can  be extended to a strong shift equivalence over $\Z_+$
from $A$ to some matrix in which  $M'$ is a principal submatrix. 
%We have $\eta^{-1}$ mapping the SFT $\eta (X_E)$ back 
%to $X_E$. 
As a consequence in our case, 
there is a
 primitive 
matrix $A$ with block form 
\[
A = 
\begin{pmatrix} A(11) & A(12) & A(13) \\ 
A(21) & A(22) & A(23) \\ 
A(31) & A(32)& A(33) 
\end{pmatrix}, 
\]  
with 
\[ 
\begin{pmatrix} 
A(22) & A(23) \\ 
A(32)& A(33) 
\end{pmatrix} =
\begin{pmatrix} 
C & 0 \\ 0 & Q 
\end{pmatrix} =
E, 
\] 
such that there is 
a topological conjugacy $h\colon  X \to X_A$
such that the following holds. Identify 
 $$X_{\left( \begin{smallmatrix} A(22) & A(23) \\ A(32) & A(33)  
\end{smallmatrix} \right)}$$ and $X_E;$ then the 
restriction  of $h$ to  $X_{D(1)}$
%=\gamma(\eta (X_E) )$ 
is   
$(\gamma \eta)^{-1}$.  It follows that  $h(W')=X_{A(22)} \subset X_A$, 
so 
$h$ maps $W$ into $X_{A(22)} $. 

%We have $\eta^{-1}$ mapping the SFT $\eta (X_E)$ back 
%to $X_E$. 
%
%\[
%(\gamma_1\eta)^{-1} \colon  \gamma_1(\eta (X_E) ) \to X_E 
%\[
%h\colon  \eta (X_E)  \to 
%X_{\left( \begin{smallmatrix} A(22) & A(23) \\ A(32) & A(33)  
%\end{smallmatrix} \right)} 
%=
%X_{\left( \begin{smallmatrix} C & 0 \\ 0 & Q 
%\end{smallmatrix} \right)} \ . 
%\] 
% Therefore $h$ maps $W'$ to 
%$X_{\left( \begin{smallmatrix} A(22) 
%\end{smallmatrix} \right)} = X_C$. 
%
%the 
%conjugacy $\eta^{-1}\colon  W' \to X_{\left( \begin{smallmatrix} C & 0 \\ 0 & Q 
%\end{smallmatrix} \right)}$
%There is a strong shift equivalence from 
%$\left( \begin{smallmatrix} C & 0 \\ 0 & Q 
%\end{smallmatrix} \right)$ which implements the
%conjugacy 
%

Let $A$ be $m\times m$. Set $A(1)=A$. 
In the order $j=1,2,\dots , k-1$ define 
$I-A(j+1)$ to be the matrix obtained from 
$I-A(j)$ by adding column $m-k+j$ of 
$I-A(j)$ to column $m-k+j+1$. For each $j$, 
this is a positive matrix equivalence giving a flow equivalence 
$\phi_j\colon  \maptorus X_{A(j)} \to  \maptorus X_{A(j+1)}$ 
which is the identity on the submapping torus $SX_M$,  where 
$$M=
\left( \begin{smallmatrix} A(11) & A(12) \\ A(21) & A(22)  
\end{smallmatrix} \right).$$ 
Set $B= A(j+k-1)$. 
Then $B(m,m) = N$  and the composition 
\[
\phi_{k-1}\cdots \phi_2\phi_1(\maptorus h) 
(\maptorus \gamma)\colon  \maptorus X \to \maptorus
X_B 
\] 
is the desired flow equivalence. 
\end{proof} 

\begin{remark} There is an alternate proof of Lemma \ref{makeroom} 
which constructs $B$ using a sequence of flow equivalence 
arguments from the proofs of 
\cite[Lemma 2.1, Corollary 2.3, Theorem 2.4]{jf:fesft}. 
\end{remark} 

We turn now to a brief review of the isotopy futures group. 
 
There is a homomorphism (the dimension representation) $\rho_A$ 
from the automorphism group of an SFT $X_A$
to the group of automorphisms of its dimension group.
An automorphism 
in the kernel of $\rho_A$ is called {\it inert}; 
it acts by the identity on the dimension group.
We are using the dimension group built from
left infinite rays; for background on this, see 
\cite{mbdldr:agsft}.

The {\it mapping class group} of an SFT $X_A$  is the group of
orientation preserving  homeomorphisms 
of its mapping torus $\maptorus{X_A}$, modulo isotopy.
There is a group associated to 
$\maptorus{X_A}$ which is the flow equivalence analogue of the 
association of the dimension group to $X_A$: the 
{\it isotopy futures group}, $\mathcal F(X_A)$. 
This group is the free abelian group with generators the set of 
rays $x(-\infty,n]$, $x\in X_A$ and $n\in \Z$, given certain relations. 
The map which sends a ray  $x(-\infty,n]$ to the vector $e_j$ such that 
$j$ is the terminal vertex of $x_n$ induces an isomorphism from 
$\mathcal F(X_A)$ to $\cok(I-A)$. 
The construction is very similar  to Krieger's construction of the
dimension group out of rays.  
There is also a flow analogue of the dimension representation: 
a flow equivalence $\mt X_A \to  \mt X_A$,
by its action on finite unions of rays, induces an automorphism 
of  $\mathcal F(X_A)$. 
An inert automorphism $U$ of 
$X_A$ induces a flow equivalence of $\maptorus{X_A}$ which 
acts by the identity on $\mathcal F(X_A)$,  
because the action of $\mt U$ on $\mathcal F(X_A)$ 
factors through the action of $U$ the dimension group. 

See \cite[Section 7]{mb:fesftpf} for 
the development of isotopy futures  theory, and
\cite{bce:fei} for more on isotopy and the mapping class
group of a shift of finite type. 

% We will \blue{appeal to background on the dimension and 
%isotopy futures groups}   in \lq\lq Moreover\rq\rq \ statements in 
%statements of some results in this section, the entire goal of which 
%is the \lq\lq Moreover\rq\rq \  statement of 
%\blue{Theorem \ref{extension-thm}, which we expect may be useful for 
%future coding arguments. We do not use the this statement in the 
%current paper.}  
%In the current paper, we only apply this for our 2-sofic classification 
%(Section \ref{2soficsection}); there, it is essential 
%to the proof.  

\section{The reduction theorem for AFT shifts}\label{sec:rt}

We are ready to state our main result which reduces the question of
flow equivalence of AFT sofic shifts to a question of flow equivalence
of certain covers.

%The reader is referred to \cite{dlbm:isdc} for the definitions of
%relevant classes of shift spaces, and properties associated to them,
% but let us briefly review the key properties of the class of almost finite type (AFT) shifts, originally due to Marcus
%\cite[Definition 4]{bm:ssed}.

\begin{defi} \label{def:AFTdef}
The shift space $X$ is said to be  \emph{almost finite type} if
there is an irreducible subshift of finite type $Y$ and
a factor map $\pi\colon Y\to X$ that is one-to-one on a nonempty open set.
\end{defi}

The AFT shifts, 
originally introduced by 
 Marcus \cite[Definition 4]{bm:ssed}
to address 
practical coding problems,  
 have emerged 
as a natural and large class of relatively tractable sofic 
shifts \cite[Sec.13]{dlbm:isdc}. AFT shifts have a variety of 
characterizations, 
collected below in Theorem 
\ref{thm:aftconditions}. 

%Collecting known results (see \cite{dlbm:isdc}, for example), we have the following list of
%equivalent conditions on a strictly sofic shift.

\begin{theorem} \label{thm:aftconditions}
Let $X$ be an irreducible, strictly sofic shift.  The following are equivalent
\begin{enumerate}[(i)]
\item  The shift $X$ is AFT.
\item  The left Fischer cover of $X$ is right-closing.
\item  The right Fischer cover of $X$ is left-closing.
\item  $X$ has a minimal cover (i.e.\ an SFT $Y$ and a factor map
  $\pi\colon Y\to X$ such that any other factor map $\phi\colon Y' \to X$ (from an SFT $Y'$ onto $X$) must factor through $\pi$) 
\cite{bkm,WilliamsMinimal}. 
(This cover must be conjugate to the 
left and right Fischer covers.) 
\item The right and left Fischer covers of $X$ are topologically 
conjugate as factor maps.  
\item $X$ is the factor of an SFT by a biclosing map. 
\item The multiplicity set of its (left or right) Fischer cover 
is a proper subshift of the domain. 
\end{enumerate}
\end{theorem}

We draw the reader's attention to (vii) in particular, since it will allow us to work with 
multiplicity sets such as $\mulsetdown(\pi)$  and
$\mulset(\pi)$ as shift spaces in their own right. 
Because the left and right Fischer covers of an AFT shift
  are 
conjugate, 
we may (when concerned only with the conjugacy class of the 
Fischer cover) refer to {\it the} Fischer cover of an AFT
shift.
%The multiplicity set 
%$\mulset(\cover(X))$ (Definition \ref{defn:multiplicity}) 
%of the Fischer cover  is a proper subshift of $\cover(X)$ if and 
%only if $X$ is AFT.} 

\begin{theorem} \label{theorem:reduction} (Reduction Theorem) 
For $i=1,2$, let 
$\pi_i \colon Y_i \to X_i$ be the Fischer 
cover of an 
%a strictly sofic 
AFT shift $X_i$. 
Then the
  following are equivalent: 
  \begin{enumerate}
  \item $X_1$ and $X_2$ are flow equivalent. \label{item:1}
  \item The two factor maps $\pi_{1}$ and
    $\pi_{2}$ are flow equivalent factor maps. \label{item:2}
  \item $Y_1$ and $Y_2$ are flow equivalent and the
    restricted factor maps $(\pi_{1})\rM$ and
    $(\pi_{2})\rM$ are flow equivalent factor maps. \label{item:3} 
  \end{enumerate}
\end{theorem}

\begin{proof}
  It is obvious that \eqref{item:2} implies \eqref{item:3}. 
 It is proven in \cite{KriegerSoficI} that the Fischer cover is canonical,
  and it follows from Proposition \ref{prop:fischer} that it also
  respects symbol expansion. The domain of  each Fischer cover 
$\pi_i$ is an irreducible SFT,  in which  $\mulset(\pi_i)$ is a
proper  subshift \cite{mbwk:amsess}. 
Thus
 all of the assumptions of
  Theorem \ref{theorem:prereduction} are satisfied, and 
\eqref{item:1} implies \eqref{item:2}. 
It remains  to prove that
  \eqref{item:3} implies \eqref{item:1}.

  Suppose then that condition \eqref{item:3} holds. 
Then there are flow equivalences $\phi, \psi$ giving  a 
commuting central square in the following diagram  
  \begin{equation*}
    \xymatrix{
\maptorus{Y_1}
\ar[ddd]_{\maptorus{\pi_{1}}}\ar[rrr]^{\widetilde{\phi}}
&&&\maptorus{Y_2}
\ar [ddd]^{\maptorus{\pi_{2}}}\\
&      \maptorus{\mulset(\pi_{1})}\ar@{_{(}->}[lu]\ar[r]^{\phi}\ar[d]_{\maptorus\pi_{1}\rM} 
      &\maptorus{\mulset(\pi_{2})}\ar[d]^{\maptorus\pi_{2}\rM}\ar@{^{(}->}[ru]&\\
&      \maptorus{\mulsetdown(\pi_{1})}\ar[r]^{\psi}\ar@{^{(}->}[ld]
      &\maptorus{\mulsetdown(\pi_{2})}\ar@{_{(}->}[rd]&\\
\maptorus{X_1}
%\ar@{-->}[rrr]
&&&\maptorus{X_2}
} 
  \end{equation*} 
in which the hooked arrows are inclusions. 
By the Extension Theorem \ref{extension-thm},  there exists a flow
  equivalence $\widetilde{\phi}\colon \maptorus{Y_1} \to 
\maptorus{Y_2}$ extending $\phi$, so that the entire diagram
commutes. 
%such that
%  $\phi(\maptorus{\mulset{\pi_{1}}})=\maptorus{\mulset{\pi_{2}}}$ and a flow equivalence $\psi$
%  from $\maptorus{\mulsetdown{\pi_{1}}}$ to
%  $\maptorus{\mulset{\pi_{2}}}$ such
%  that the diagram 
%  commutes.
Because the homeomorphism $\widetilde{\phi}$ takes the quotient
relation 
of $\mt{\pi_1}$ to that of $\mt{\pi_2}$, it induces a homeomorphism
  $\maptorus{X_1} \to \maptorus{X_2}$, which is easily seen to be 
a flow equivalence. 
%  It follows that $\maptorus{\pi_{1}}(x)\mapsto
%  \maptorus{\pi_{2}}(\phi(x))$ induces a well-defined map from
%and it is straightforward to
%  check that it is a flow equivalence. Thus $T_1$ and $T_2$ are flow
%  equivalent. 
Hence \eqref{item:3} implies \eqref{item:1}.
\end{proof}

%
%Specializing to the AFT case and invoking more succint notation, we get
%
%\begin{corollary}\label{maincor}
%  Let $X_1$ and $X_2$ be strictly sofic AFT shifts. For
%  $i\in\{1,2\}$ let $(Y_i,\pi_i)$ be the (left or right) Fischer cover of $X_i$. Then the following are equivalent
%  \begin{enumerate}
%  \item $\FE{X_1}=\FE{X_2}$
%  \item $\FE{\pi_1}=\FE{\pi_2}$
%  \item $\FE{Y_1}=\FE{Y_2}$ and $\FE{\pi_1\rM}=\FE{\pi_2\rM}$
%  \end{enumerate}
%\end{corollary}
%
%\begin{proof}
%  It is proven in \cite{MR776312} that the Fischer cover is canonical,
%  and it follows from Proposition \ref{prop:fischer} that it also
%  respects symbol expansion. A sofic AFT shift is by definition
%  irreducible and it is shown in \cite{mbwk:amsess} that $\mulset{\pi_i}$ is a
%  subshift of $Y_i$ for $i\in\{1,2\}$. Thus all of the assumptions of
%  Theorem \ref{theorem:reduction} are satisfied, and the corollary follows.
%\end{proof}

%In many cases, the fiber products are useful invariants for telling
%irreducible sofic shifts apart, even when they fail to be AFT; which
%as we now know, is the same as having the diagonal component
%isolated. 

%\section{Near Markov shifts}

%In this section we use the Reduction Theorem \ref{theorem:reduction}
Next, as an immediate application of the Reduction Theorem \ref{theorem:reduction}
we classify 
the irreducible 
near Markov  shifts up to 
flow equivalence.

\begin{defi} 
\cite{mbwk:amsess} An irreducible  sofic shift space $X$ is \emph{near
  Markov} if one of its Fischer covers $\pi \colon Y\to X$ has a finite multiplicity set $\mulset(\pi)$.
\end{defi}

A near Markov shift is AFT, so as noted above, 
we can refer to {\it the} Fischer cover (up to conjugacy of factor maps).
%   By Theorem
%\ref{theorem:prereduction} 
%we now easily get

\begin{prop}\label{nearminv}
If $X$ is flow equivalent to $X'$, and $X$ is 
irreducible near Markov, then so is $X'$.
\end{prop}
\begin{proof} 
Because $X$ is an AFT shift, 
so is $X'$
\cite{FujiwaraOsikawa}. 
The irreducible near Markov shifts are precisely 
the AFT sofic shifts for which the 
mapping torus of the multiplicity 
set $\mulset (\pi )$ of the Fischer cover $\pi$ 
is a finite union of circles, so the conclusion follows 
from  
Corollary \ref{corollary:prereduction}.

%The class of finite SFTs is flow equivalent 
%since these are precisely the shift spaces having mapping tori homeomorphic to a finite union of circles. Since we know by Theorem \ref{theorem:prereduction} that the multiplicity sets of $T$ and $T'$ are flow equivalent, the conclusion follows.
\end{proof}

\begin{defi}\label{def:perpoint}
  The \emph{multiplicity graph} $\mugr{\pi}$ of a factor map $\pi\colon Y\to
  X$ with $\mulset(\pi)$ finite is a bipartite graph defined as
  follows. Organize the finite sets $\mulset(\pi)$ and
$\mulsetdown(\pi)$
into orbits $o_1,\dots o_k$ and
 $v_1,\dots ,v_\ell$ and 
let $j(i)$ be the $j$ such that $\pi(o_i)=v_{j}$. 
Note that in this case, $|v_{j(i)}|$ must divide $|o_i|$, and set
\[
w(i)=\frac{|o_i|}{|v_{j(i)}|}.
\]
The graph $\mugr{\pi}$ then has vertices $\{y_1,\dots,y_k,x_1,\dots x_\ell\}$ with $w(i)$ edges from $y_i$ to $x_i$ for each $i\in\{1,\dots k\}$.
\end{defi}

\begin{lemma}\label{mulgraphiscool}
The multiplicity graph is a complete flow invariant for the 
class of maps  $\pi\rM (\pi)$  such that $\mulset(\pi)$ is finite.
\end{lemma}
\begin{proof}
The mapping tori for finite shifts  
$\mulset (\pi)$ and $\mulsetdown(\pi)$ 
%the two spaces 
are finite unions of circles. 
A circle $\mt o_i$ 
in  $\maptorus \mulset (\pi)$ is wrapped by $\maptorus\pi$ 
$w(i)$  times around its image circle
$\mt v_{j(i)}$;
  this (winding number) $w(i)$ is a complete
invariant 
of flow equivalence of the map $\maptorus\pi|_{\maptorus o_i}$.
% and the map among these sets restricted to the covering circles have
% winding numbers given by the multiplicities in the graph.
 Hence 
%all the information in 
the multiplicity graph  encodes an invariant of flow equivalence, 
which is easily checked to be complete,  for this class.
The conclusion follows by Theorem \ref{theorem:reduction}.
\end{proof}

\begin{defi}
  For a near Markov shift $X$ with Fischer cover $\pi\colon Y\to X$ where $Y=Y_A$ for a matrix $A$, we define $\inva(X)$ as the collection of data
\[
\left(\BFgroup{Y_A}, \det (I-A) , \mugr{\pi}\right).
\]
\end{defi}

\begin{theorem}\label{nearmarres}
  For a pair of near Markov shifts $X,X'$ we have
\[
\FE{X}=\FE{X'}\Longleftrightarrow \inva(X)\simeq \inva(X').
\]
\end{theorem}

\begin{proof}
We may apply Theorem \ref{theorem:reduction} since $\mulset(\pi)$ is closed by Theorem \ref{thm:aftconditions}(vii). Now the first component of our invariant is a complete invariant of flow equivalence of the Fischer cover by Franks's classification (\cite{jf:fesft}), and the latter is a complete invariant of the multiplicity cover by Lemma \ref{mulgraphiscool}.
\end{proof}

Franks proved in \cite{jf:fesft} that any irreducible SFT is flow
equivalent to its time-reversal, by noting that his complete invariant
did not distinguish them.  
This fails dramatically  even for general AFT shifts; e.g., for the Fischer 
cover $\pi$, we could arrange $\mulsetdown (\pi )$ to be the SFT 
presented by the matrix $\left(\begin{smallmatrix} 2 & 1 \\ 0 & 1
\end{smallmatrix}\right)$. However, Franks' result  
carries over to near Markov shifts.
\begin{corollary}
A near Markov shift is flow equivalent to its time-reversal.
\end{corollary}
\begin{proof}
The complete invariant is the same for the system and its time-reversal.
\end{proof}

\section{$N$-point extensions and $G$-SFTs} \label{sec:np}
\label{sec:constant} 

For a set $E$, let $\symE$ be the group of permutations
of $E$, with the group product $gh$  defined by
$gh \colon x\mapsto g(h(x))$ (i.e., $h$ acts first).  
Let $\symN$ denote $\symE$ with
$E= \{1, \dots , N\}$. 

In this section we recall how to
reduce the  classification up to
topological conjugacy of  
$N$-point extensions of 
SFTs  to the 
classification of related $G$-SFTs with 
$G=\symN$. 
%,the group of permutations of $\{1, \dots N\}$.
This reduction will be used for flow 
equivalence results in Section \ref{sec: pettype}. 
The reduction is  due to 
 Adler, Kitchens and
 Marcus \cite{akmfactor,akmgroup}
 (adapting ideas of Rudolph
 %who followed
 %the reduction 
 %introduced by Rudolph for ergodic theory
\cite{rudolphcounts}).  

Recall that factor maps $\pi\colon  Y\to X$ and $\pi'\colon  Y' \to X'$ 
are defined to be isomorphic (topologically conjugate)
 if there are topological conjugacies 
$\alpha, \beta$ such that $\pi'\alpha = \beta \pi$. 
Equivalently, there is  a topological conjugacy $\alpha \colon  Y\to Y'$ 
such that for all $w,y$ in $Y$: 
$\pi (w) = \pi (y)$ if and only if $\pi'(\alpha (w)) = \pi'(\alpha
(y))$.

\begin{definition} 
\label{defn:npoint}
Suppose  $N$ is a positive integer and 
$\sigma\colon  X\to X$ is a homeomorphism of a compact metric space. 
%$(X,\sigma) $ is a topological dynamical system. 
Let $Y=X\times \{1, \dots ,N\}$ and $\rho\colon Y\to Y$ be a homeomorphism  of the form 
\[
\rho\colon (x,k)\mapsto (\sigma(x), \tau_x(k)),
\]
with 
$\tau_x\in \symN$
acting from the left 
(e.g., $\rho^2\colon (x,k)\mapsto (\sigma^2(x), \tau_{\sigma(x)}\tau_x(k))$).
Then the factor map $(Y,\rho)\mapsto (X,\sigma)$ defined 
by $(x,k)\mapsto x$ is an \emph{$N$-point extension} of 
$(X,\sigma)$. 
A factor map $(Y,\rho)\to (X,\sigma)$ is an $N$-point 
extension of $(X,\sigma)$ if and only if it is isomorphic 
to such a factor map.
When $(Y,\rho)$ is the domain system of
an $N$-point extension of $(X,\sigma)$, $(Y,\rho)$ itself is
sometimes referred to as an $N$-point extension of $(X,\sigma)$. 
The function $\tau$ is called the {\it skewing function}. 
\end{definition}

\begin{facts} \label{routine}
We mention some routinely verified facts. 
\begin{enumerate} 
\item 
The map $ x\mapsto \tau_x$  from Definition \ref{defn:npoint} is 
a continuous function $\tau \colon  X\to \symN$.   
\item 
A continuous, 
 constant \ntoone{N} 
factor map $\pi \colon  (Y,\rho)\to (X,\sigma)$ is an 
 $N$-point extension of $(X,\sigma)$ if and only if 
there are $N$ 
disjoint sections to $\pi$, i.e.\ continuous maps 
$\tau_i\colon  X\to Y$, $1\leq i \leq N$,
with disjoint images, satisfying $(\pi\tau_i) (x) = x$ 
for all $x$ in $X$. 
\item 
An $N$-point extension of $(X,\sigma)$ is SFT if and 
only if $(X,\sigma)$ is 
SFT.  
\end{enumerate} 
\end{facts} 
\begin{example} 
A constant  \ntoone{N} factor map of SFTs need not be 
an $N$-point extension; for example, the 
matrix $\widetilde C$ below is the adjacency matrix of a 
labeled graph for which the 
labels define a one-block code of edge SFTs,  $X_C \to X_{\overline
  C}$,  which is 
constant \ntoone{2} but is not a $2$-point extension. 
\[
\widetilde C = 
\begin{pmatrix} 
a& 0& b& 0 \\
0& a& b& 0 \\
0& 0& c& 0 \\
0& 0& 0& c 
\end{pmatrix} 
\qquad   
C = 
\begin{pmatrix} 
1& 0& 1& 0 \\
0& 1& 1& 0 \\
0& 0& 1& 0 \\
0& 0& 0& 1 
\end{pmatrix} \qquad 
\overline C = 
\begin{pmatrix} 
1& 1 \\
0& 1
\end{pmatrix}. 
\]
\end{example}

In contrast, we have the following
key fact,   
which follows immediately  from 
Nasu's work
\cite{NasuConstant}
(cited in \cite[p.\ 489, Remarks (iv)]{akmfactor}) 
after a translation of terminology. 
Tools for a  proof (not a stated result) can also be
found in \cite[Sec. 4.3]{bpk:sd}.

\begin{theorem} \cite[Theorem 7.3; see also Corollary 6.6]{NasuConstant}
  Suppose $\pi $ is a constant \ntoone{N}  factor map
  between irreducible SFTs. Then $\pi$ is an $N$-point extension.
\end{theorem}

We define two $N$-point extensions to be isomorphic 
if they are 
 isomorphic as  factor maps. 
%
%
%%% Hopefully we don't need the following: 
% 
% The term  \lq\lq$N$-point extension of $(X,S)$\rq\rq\ 
%  is also  
% used to refer to the topological dynamical system 
% $(Y,T)$ in our Definition \ref{defn:npoint} -- i.e., to the 
% domain system of a factor map as described in Definition
% \ref{defn:npoint} -- with the understanding that 
% isomorphism is topological conjugacy intertwining 
% the associated $N$ to one factor maps.  
The following standard fact is another routine exercise. 
\begin{fact} \label{fact:npoint}
%Two 
$N$-point extensions,  
defined by data $(X,\sigma, \tau ) $  and 
 $(X',\sigma', \tau' ) $ as in 
Definition
\ref{defn:npoint}, with $\tau $ and $\tau' $ skewing from the 
left,   are isomorphic if and only if there 
is a conjugacy $\phi\colon  (X,\sigma) \to (X',\sigma')$ 
such that $\tau$ and $(\tau'\phi)$ are 
cohomologous in $(X,\sigma)$: i.e., 
there is 
 a continuous $\gamma \colon  X \to \symN$ such that 
for all $x$ in $X$, 
\[
\tau' ( \phi (x))\  =\  
\gamma ( \sigma (x))\,  \tau (x)\, 
(\gamma (x)  )^{-1}
\]
 (where the right hand side is a product in the 
group $\symN$).   
\end{fact} 

\begin{definition} In this paper,  a  $G$-SFT is a 
shift of finite type $X$ together with a free, continuous shift-commuting
action of a finite group $G$. The factor map of the $G$-SFT 
is the everywhere \ntoone{|G|} map which collapses $G$-orbits to 
points. The $G$-SFT is a left $G$-SFT if the $G$-action is a left
action ($g\colon  x\mapsto gx$; $gh\colon x \mapsto g(hx)$);
it is a right  $G$-SFT if the $G$-action is a right 
action ($g\colon  x\mapsto xg$; $gh\colon x \mapsto (xg)h$).
%If not specified, a $G$-SFT is a left $G$-SFT. 
\end{definition}

By definition, a  conjugacy (isomorphism)
of two $G$-SFTs is a topological conjugacy   of the underlying SFTs
which intertwines their $G$-actions; equivalently, it is a conjugacy 
of the factor maps  of the $G$-SFTs
which intertwines their $G$-actions. 

The factor map of a right $G$-SFT is a $|G|$-point extension, and therefore can be presented as in Definition \ref{defn:npoint}, with $G$ in place of the set $\{1,\dots , N\}$. Here the permutation $\tau_x$ of   Facts \ref{routine} must be left multiplication by some element $\beta (x)=\beta_x$ of $G$, as the right $G$-action commuting with the shift forces for $g$ in $G$ that $  \tau_x(eg) =  \tau_x(e)g$. 

\begin{fact} \label{fact:rightgsft}
Two right $G$-SFTs $(X,\sigma)$ and $(X',\sigma')$ with right $G$-actions $\beta$ and $\beta'$ are isomorphic if and only if there 
is a conjugacy $\phi\colon  X \to X'$ and 
 a continuous $c \colon  X \to G$ such that 
for all $x$ in $X$, 
\begin{equation} \label{cohomologyeqnforrightactionofG} 
\beta' ( \phi (x))\  =\  
c (\sigma(x))   \, \beta (x)\, c ( x)^{-1}. 
\end{equation}
(where the right hand side is the product in the group $G$). 
\end{fact}
Fact \ref{fact:rightgsft} holds because the  permutation
$\gamma (x)$ in Fact \ref{fact:npoint} is here an
element of $\symG$ which 
commutes with the right $G$-action, and again  must be left multiplication
by some element $c(x)$ of $G$.

As noted in \cite{akmfactor}: every 2-point extension of SFTs is 
isomorphic to  the factor map of some  $G$-SFT
with $G=\Z/2\Z$  
 (Remark \ref{extensioncontext} gives one proof), but 
for $N >2$, 
an $N$-point extension of an SFT is not in general 
isomorphic to the factor map of a $G$-SFT. For example,  
if $\pi\colon  (Y,\rho)\to (X,\sigma)$ is the factor map of a $G$-SFT, 
then two $\rho$-periodic points with the same image must have the 
same $\rho$-period; but a 
3-point extension could collapse a fixed point and an orbit of 
size 2 to a fixed point. 

Nevertheless, the classification of $N$-point extensions 
of SFTs can be reduced to the classification of $G$-SFTs.

\begin{definition} \label{defn:fullextension}
  \cite[p.\ 493]{akmfactor} 
The {\it full extension} of   an $N$-point 
extension of 
%an SFT 
a system 
$(X,\sigma)$, presented as above by 
$\tau \colon X\to \symN$, is the
% right $G$-SFT (with $G=\symN$) which is the 
self map of  
$ X\times \symN$ defined by the rule 
$(x, \alpha) \mapsto (\sigma(x), \tau(x) \alpha )$, 
(here 
$\tau(x) \alpha$ is the product in $\symN$) 
with right $\symN$-action  $h\colon (x,g)\mapsto (x,gh)$.    
\end{definition} 

%Here we say 
%$Y$ is a
% {\it skew product} over $X$ if 
%
%there is a continuous function 
%$\tau\colon  X \to \mathcal S_N$ such that $Y=X\times \{ 1, \dots ,N\}$ 
%and there is $\tau \in C(Y, \mathcal S_n) $ such that 
%$T$ is defined by $T\colon  (x,k) \mapsto (\sigma x, (\tau (x))(k) )$. 
%We say then that $\pi \colon  (x,k) \mapsto x$ 
%is the factor map of a skew product. In this case, 
%without loss of generality we may (after passing to a 
%higher block presentation of $X$), assume that $\tau (x)$ 
%depends only on $x_0$. 

%\begin{proposition} \label{goodconstantmap} 
%Suppose $\pi$ is a constant  $ N $-to-1 
%factor map from an SFT $X$ onto a subshift $Y$. 
%Then the following are equivalent. 
%\begin{enumerate} 
%\item $\pi$  is e and u biclosing.  
%\item $ pi$ has a recoding to one block e and u biresolving factor
%  map. 
%\item 
%$\pi $ is conjugate to the factor map of a skew product. 
%\end{enumerate} 
%Moreover, if  $X$ is an irreducible SFT, then 
%these equivalent conditions always hold. 
%\end{proposition} 
%
%\begin{proof} 
%(1) $\iff$ (2), and likewise the ``Moreover'' claim, were  
%proved by Kitchens and Nasu; proofs can be found 
%in   (see \red{reference in Kitchens book; 
%add comment if they were only stated for irreducibles. 
%Also ref Parry's notes? Parry-Tuncel? }).  
%\end{proof} 

\begin{proposition} \label{skewconjugate}
Let $\pi\colon  (Y,\rho) \to (X,\sigma)$ 
and $\pi'\colon  (Y',\rho') \to (X',\sigma') $ be  $N$-point extensions of two 
shifts of finite type $X$ and $X'$, presented by 
skewing functions $\tau,\tau'$ acting from the left as in
Definition \ref{defn:npoint}. 
Then the following are equivalent. 

\begin{enumerate} 
\item 
$\pi$ and $\pi'$ are conjugate factor maps.  

\item 
 There is a conjugacy $\phi \colon ( X,\sigma) \to (X',\sigma')$ 
and  a continuous $\gamma \colon  X \to \symN$ such that 
for all $x$ in $X$, 
\[
\tau' ( \phi (x))\  =\  
\gamma ( \sigma (x))\,  \tau (x)\, 
(\gamma (x)  )^{-1}.
\]

\item 
 The associated full extensions are 
conjugate right $\symN$-SFTs.
\end{enumerate} 
\end{proposition} 

\begin{proof} 
$(1)\iff (2)$  is Fact \ref{fact:npoint}, and 
$(3) \iff (2)$  is Fact \ref{fact:rightgsft}. 
\end{proof}

In  \cite[Theorem 4.2(B,C)] {akmfactor}, 
Adler, Kitchens and Marcus provided  easily computed
{group invariants
for almost topological conjugacy, and hence conjugacy,} 
  of
certain 
full extensions (the nonwandering $\mathcal S_N$-trans\-i\-tive
extensions){. (For examples of their use to distinguish 
 $N$-point extensions,  see \cite[p.\ 258]{rudolphcounts}).} 
More difficult algebraic invariants are 
required for a classification of $\mathcal S_N$-extensions 
up to conjugacy, or up to flow
equivalence \GFE.

\begin{remark} \label{extensioncontext} \cite[p.\ 494]{akmfactor}
  Let a full extension $\widetilde{\pi}\colon X\times \symN \to X$
  of an $N$-point extension
  $\pi\colon X \times \{1, \dots , N\}$ 
  be given using $\tau$ as in Definition  
  \ref{defn:fullextension}. Then the map
  $\beta\colon  X\times \symN \to X \times \{1, \dots , N\}$
  given by 
  $(x,g) \mapsto (x,g(1))$ is an $(N-1)!$-point extension and
  $\widetilde{\pi} = \pi\beta$ .  The map $\beta$
  is isomorphic to the map obtained by using in place of  $g(1)$
   the coset $gH$, where $H=\{h\in \symN \colon  h(1)=1 \}$. 
\end{remark}

\begin{remark}
  There are some differences between our presentation
  and terminology and what's in \cite{akmfactor,akmgroup}.
  We have only taken some of their beginning content --
  the papers were concerned with almost topological
  conjugacies of factor maps. 
  Some of our statements are only implicit in
  \cite{akmfactor,akmgroup}.
  The seminal measurable version of  Proposition \ref{skewconjugate}
  is explicit in Rudolph's paper \cite[Lemma 1]{rudolphcounts}.
  \end{remark}

More background on $G$-SFTs 
can be found in  \cite{akmfactor,akmgroup,BSullivan}
and \cite[Appendix A]{BoSc2}. 
The chosen action in 
\cite{BSullivan} should be a left rather than right action, 
as explained in \cite[Appendix A]{BoSc2}.

\section{PET sofic shifts} \label{sec:pet}

In this section, we will 
use full extensions to reduce the 
FE classification of a certain class of 
AFT shifts to the flow equivalence  classification of 
 $G$-SFTs, for which  complete invariants 
are known \GFE.  First, we must address a technical 
point involving left vs. right actions.

A square matrix $A$ over $\Z_+G$ presents a $G$-SFT
(equivalently, a $G$ extension of an SFT) in a
natural way: 
the matrix $A$ gives  a labeling of edges of
a directed graph by elements of $G$, say $e\mapsto \ell (e)$.
The graph defines the usual edge SFT and the labeling
defines a skewing function $\tau_x$: an element of $G$ 
is multiplied by $\ell (x_0)$. This is a left $G$-SFT if 
$\tau_x \colon g \mapsto g\ell (x_0) $. 
The {\it left} $G$-SFTs
presented by $A$ and
$B$ are topologically conjugate if and only if
the matrices $A$ and $B$ are strong shift equivalent
over $\Z_+G$ 
(see \cite{BSullivan}
and \cite[Appendix A]{BoSc2}). 
This leads us to a natural definition. 

\begin{definition} 
The {\it left full extension} of an
$N$-point extension is defined as 
in Definition \ref{defn:fullextension}, with 
the following changes: $\tau$ is chosen to act from the right, 
and then $\symN$ is taken to act from the left. 
(So, the left full extension of an $N$-point extension 
of an SFT is a left $G$-SFT, with $G=\symN$.) 
\end{definition}

%The key Proposition \ref{skewconjugate}
%uses full extension which are {\it right}
%$G$-SFTs, corresponding to the use of
%skewing function acting from the left,  
%as in \cite{akmfactor,akmgroup,rudolphcounts}. 
%However we may instead require the skewing function 
%to act from the right; then the corresponding 
%full extension is chosen to be a {\it left  $G$-SFT.} 
Fact \ref{routine}(2) does not distinguish between left 
and right; the same class of extensions (up to 
topological conjugacy) is presented with functions 
skewing from the right as for functions skewing 
from the left. 
Arguments for left $G$-SFTs mimicking
those for right $G$-SFTs in the last section
then lead to the following analogue of Proposition
\ref{skewconjugate}, with an additional condition (4).

\begin{proposition} \label{skewconjugateleftaction}
Let $\pi\colon  (Y,\rho) \to (X,\sigma)$ 
and $\pi'\colon  (Y',\rho') \to (X',\sigma') $ be  $N$-point extensions of two 
shifts of finite type $X$ and $X'$, presented by 
skewing functions $\tau,\tau'$ as in
Definition \ref{defn:npoint},
but with $\tau$ and $\tau'$ skewing from the right.  
Then the following are equivalent. 

\begin{enumerate} 
\item 
$\pi$ and $\pi'$ are conjugate factor maps.  

\item 
 There is a conjugacy $\phi \colon ( X,\sigma) \to (X',\sigma')$ 
and  a continuous $\gamma \colon  X \to \symN$ such that 
for all $x$ in $X$, 
\[
\tau' ( \phi (x))\  =\  
(\gamma (x)  )^{-1}\,  \tau (x)\, 
\gamma ( \sigma (x)). 
\]

\item 
 The associated full extensions are 
conjugate left $\symN$-SFTs. 
\item 
If the left $\symN$-SFT full extensions are presented by matrices 
$A,B$ over $\Z_+\symN$, then $A$ and $B$ are strong shift equivalent 
over $\Z_+\symN$. 
\end{enumerate} 
\end{proposition} 

If one remains with a presentation of an 
$N$-point presentation with $\tau$ skewing from the left, 
one can still reduce to a SSE-$\Z_+$ invariant. Let $A$ and $B$ 
be $m\times m$ matrices over $\Z_+G$ presenting full extensions for 
$N$-point presentations $\pi, \pi'$ with $\tau$ skewing from the left. 
As in \cite[Appendix A]{BoSc2}, 
let  $A^{\text{opp}}$ to be  the $m\times m$ matrix such that 
$A(i,j)  = \sum_g n_g g \implies A^{\text{opp}} (i,j)  = \sum_g n_g
g^{-1}$, and likewise define   $B^{\text{opp}}$. Then $\pi$ and $\pi'$ 
will be isomorphic $N$-point extensions if and only if  
$A^{\text{opp}}$ and $B^{\text{opp}}$ are SSE-$\Z_+$ (i.e., 
define isomorphic left $\symN$-SFTs).   

The condition (2) in Proposition \ref{skewconjugateleftaction} 
reflects the following analogue of 
Fact  \ref{fact:rightgsft}. 

\begin{fact} \label{fact:leftgsft}
Two left $G$-SFTs $(X,\sigma)$ and $(X',\sigma')$ with left $G$-actions $\beta$ and $\beta'$ are isomorphic if and only if there 
is a conjugacy $\phi\colon  X \to X'$ and 
 a continuous $c \colon  X \to G$ such that 
for all $x$ in $X$, 
\begin{equation} \label{cohomologyeqnforleftactionofG} 
\beta' ( \phi (x))\  =\  
(c (x)  )^{-1} \, \beta (x)\, c ( \sigma (x))
 \  
\end{equation}
(where the right hand side is the product in the group $G$). 
\end{fact}

It is elementary but important to note 
that for a nonabelian group $G$, 
the cohomology equations 
\eqref{cohomologyeqnforrightactionofG} and 
\eqref{cohomologyeqnforleftactionofG} 
%\eqref{cohomologyinG} 
are not equivalent. 
Here is a simple example  (distilled from \cite[Example A.4]{BoSc2}). 

\begin{example}\label{leftvsrightexample}
Let $G$ be nonabelian. Let $a,b$ be group elements such that $ab\neq ba$.
Let $d=(ab)^{-1}$; then $abd = e \neq dba$.
 Let $X=\{x, \sigma (x), \sigma^2(x)\}$, a single orbit containing three
 points. 
Define $\beta(x)=a,\beta(\sigma(x))=b, \beta(\sigma^2(x))=d$ and
$\beta'(x)=\beta'(\sigma(x)) = \beta'(\sigma^2(x))=e$.  Then there is a 
function $c(x)$
satisfying 
\eqref{cohomologyeqnforrightactionofG},  
but there is no
$c(x)$ 
satisfying \eqref{cohomologyeqnforleftactionofG}. 
\end{example}

We now turn to flow equivalence. If $X$ is a left $G$-SFT, then the left $G$-action induces a left $G$-action on $\mt X$. Two left $G$-SFTs are $G$-flow equivalent if there exists a flow equivalence $\psi\colon \mt X\to\mt X'$ which intertwines the $G$-actions.

\begin{proposition} \label{feequiv}
Suppose 
$\pi$ and 
$\pi'$ are 
$N$-point extensions of SFTs.
Then the following are equivalent. 
\begin{enumerate} 

\item 
As factor maps, $\pi$ and $\pi'$ are flow equivalent. 

\item 
The left full extensions of $\pi$ and $\pi'$  
are
$\symN$-flow equivalent  left 
$\symN$-SFTs. 
\end{enumerate} 
\end{proposition} 

\begin{proof} 
We will first prove the implication $(1)\implies (2)$. For this we define \emph{symbol expansion} for $N$-point extensions. Let $\pi\colon (Y,\rho)\to (X,\sigma)$ be an $N$-point extension of a shift space, presented by the skewing function $\tau$ as in
Definition \ref{defn:npoint}, but with $\tau$ skewing from the right, and with $\pi\colon Y\to X$ a one-block code. 
Let $a\in\al(X)$, let $\star$ be a symbol which does not belong to $\al(X)$, and let 
$\tau_\star\in \symN$. Then the $N$-point extension 
$\expansion{a,\tau_\star}{\pi}\colon (\expansion{a,\tau_\star}{Y},\expansion{a,\tau_\star}{\rho})\to (\expansion{a,\tau_\star}{X},\expansion{a,\tau_\star}{\sigma})$ where 
$\expansion{a,\tau_\star}{X}=\expansion{a}{X}$ (cf. Definition \ref{def:expansion} and 
the remark just after the proof of Lemma \ref{lemma:1}), 
$\expansion{a,\tau_\star}{\sigma}=\sigma_{\tilde{\al}}$, 
$\expansion{a,\tau_\star}{Y}=\expansion{a,\tau_\star}{X}\times \{1,\dots,N\}$, 
$\expansion{a,\tau_\star}{\rho}\colon \expansion{a,\tau_\star}{Y}\to\expansion{a,\tau_\star}{Y}$ is given by $\expansion{a,\tau_\star}{\rho}((x,k))=(\expansion{a,\tau_\star}{\sigma}(x),\expansion{a,\tau_\star}{\tau}_x(k))$ with $\expansion{a,\tau_\star}{\tau}_x=\tau_{x'}$ if $x=\iota_a(x')$, and $\expansion{a,\tau_\star}{\tau}_x=\tau_\star$ if $x=\sigma_{\tilde{\al}}(\iota_a(x'))$, and $\expansion{a,\tau_\star}{\pi}\colon \expansion{a,\tau_\star}{Y}\to\expansion{a,\tau_\star}{X}$ is given by $\expansion{a,\tau_\star}{\pi}(x,k)=x$, is called a symbol expansion of $\pi\colon (Y,\rho)\to (X,\sigma)$. An argument similar to the one used in the proof of the Parry-Sullivan Theorem \cite{parrysullivan} shows that flow equivalence of $N$-point extensions of shift spaces is generated by conjugacy and symbol expansions. It follows from Proposition \ref{skewconjugate} that conjugacy of $N$-point extensions gives conjugacy of the full extensions, and it is easy to check the full extensions of $\pi\colon (Y,\rho)\to (X,\sigma)$ and $\expansion{a,\tau_\star}{\pi}\colon (\expansion{a,\tau_\star}{Y},\expansion{a,\tau_\star}{\rho})\to (\expansion{a,\tau_\star}{X},\expansion{a,\tau_\star}{\sigma})$ are flow equivalent. It follows that $(1)\implies (2)$.

We will now prove $(2)\implies (1)$. If $\pi\colon (Y,\rho)\to (X,\sigma)$ is an $N$-point extension, then we can recover $\maptorus{Y}$, $\maptorus{X}$, and $\maptorus{\pi}\colon  \maptorus{Y}\to \maptorus{X}$ from $\maptorus{(X\times\symN)}$ as in Remark \ref{extensioncontext}. It follows that $(2)\implies (1)$.
	
%In both cases the flow equivalence relation is generated by 
%conjugacy and symbol expansions. A symbol expansion 
%in either case respects flow equivalence. 
%Therefore the result follows from 
%Proposition \ref{skewconjugate}. 
\end{proof} 

\begin{remark} \label{gfeinvariants} 
For a finite group $G$, every $G$-SFT can be presented by 
a square matrix over $\Z_+G$ as described above. 
Complete algebraic invariants for $G$-flow equivalence 
are known, by \cite{BSullivan} in the case the extension is mixing 
and by \GFE{} in general. 
%There is a decision procedure for determining whether $G$-SFTs 
%are flow equivalent. \red{We hope -- that is not so far in 
%\cite{bs:decide}.} 
%The range of the invariant 
%is also known. 
For further discussion see 
\cite{bce:gfe,BoSc2}. 
\end{remark}

\label{sec: pettype}
%Working in the measurable category, Rudolph 
%\cite{rudolphcounts} 
%used full extensions to classify $N$-point extensions 
%of Bernoulli shifts.  
%Adler, Kitchens and Marcus  \cite{akmgroup}   produced a parallel 
%theory in the topological category 
%to classify $N$-point extensions of irreducible 
%shifts of finite type 
%up to almost topological conjugacy,   
%as a 
%tool for their subsequent classification of 
% factor maps between 
%SFTs up to almost topological conjugacy  
%\cite{akmfactor}. 

\begin{definition} \label{defn:pet}
An irreducible sofic  shift  
is {\it point extension type} (PET) if 
it has a  Fischer cover $\pi$ 
such that 
for each $k$ in 
$\mathsf{MultiCard}(\pi )$, the set
$\mulsetdown_k (\pi) $ is a 
closed (and hence a subshift).  
\end{definition} 

In Definition \ref{defn:pet}, it would be equivalent to require 
each $\mulset_k (\pi) $ to be closed (hence a subshift).
Note 
that among irreducible shifts, a PET sofic shift must be AFT
(by Theorem \ref{thm:aftconditions}{\it (vii)}), and a near Markov shift must trivially be PET, with both inclusions proper.
We will justify the PET name with Lemma \ref{petlemma}, whose proof   
%the condition is vacuously true when $|\mathsf{MultiCard}(\pi )|=1$. 
appeals to the following result of Jung. 
\begin{fact} \cite{JungOpen} \label{biclosingfact} 
A constant \ntoone{k} biclosing factor  map between  
subshifts is a $k$-point  extension. 
\end{fact}  
Nasu \cite{NasuConstant} proved  Fact \ref{biclosingfact} in the case 
the subshifts are irreducible SFTs. 
The general result is contained in
\cite[Prop. 4.5]{JungOpen}.

\begin{lemma} \label{petlemma}
Suppose $Y$ is an irreducible PET sofic shift 
with Fisher cover $\pi \colon  X_A \to Y$. 
Then for each 
$k$ in $\mathsf{MultiCard}(\pi )$, 
the restriction
$\pi \colon 
\mulset_k (\pi) \to 
\mulsetdown_k (\pi) $ 
is a $k$-point extension of a 
shift of finite type. \end{lemma}  

\begin{proof} 
For $k \in\N$ we define 
\begin{equation} \label{Epikref}
E_k(\pi ) = E(\pi )\cap \big(    \mulset_k(\pi ) \times \mulset_k(\pi )\big), 
\end{equation} 
and let
\begin{equation} \label{Epiref}
    E(\pi ) = \{(w,x)\in X_A\times X_A\colon  \pi (w) = \pi(x) \}.
\end{equation}
Then $E_k(\pi)\ne\emptyset$ only if $k\in \mathsf{MultiCard}(\pi )$ or $k=1$, and $E_1(\pi)= \{ (x,x) \colon x\in X_A \} $. 
The set $E( \pi )$ of $\pi$ is 
a shift of finite type in $X_A\times X_A$, and equals the disjoint union of 
the shifts $E_k ( \pi )$. 
Therefore each 
$E_k ( \pi )$ is also SFT. 
For $k\in \mathsf{MultiCard}(\pi ) $,  
define $p_k\colon  E_k ( \pi )\to \mulset_k (\pi )$ 
by $p_k\colon (x,y) \mapsto x$. 
Then $p_k$ is 
biclosing
(because $\pi$ is biclosing), and 
everywhere \ntoone{k}. Then by Fact 
\ref{biclosingfact} , $p_k$ is a $k$-point 
extension. Because $E_k(\pi)$ is SFT, it follows 
that $\mulset_k (\pi )$ is SFT. Then the  same 
argument applied to the restriction  
$\pi \colon
\mulset_k (\pi) \to 
\mulsetdown_k (\pi) $  shows this 
map is a $k$-point extension of a shift of finite type. 
\end{proof} 

%%%%%% Below is a discussion of alternate proofs and comments for the lemma.
%
%% 
%% This was proved in the irreducible case 
%% by Nasu REFERENCE . 
%
%
%%%%  I drop the following presketch of an alternate proof: 
%% One could alternately deduce it by applying 
%% the standard recoding (e.g. \cite[Sec. 9.1]{dlbm:isdc} or 
%% \cite[Sec. 4.3]{bpk:sd}) of a right resolving, left closing 
%% map (from a Fischer cover of the AFT shift) to a 
%% right resolving, left resolving map for which each 
%% symbol has at most $N$ preimages. The constant 
%% to 1 condition then implies every symbol has exactly 
%% $N$ preimage symbols, and the map is resolving 
%% in the sense of existence as well as uniqueness.) 
%% argument (
%
%
%%%% irrelevant: 
% f biclosing constant to one from SFT to sofic 
% implies f is open (Jungopen, Theorem 4.4 and remark 4.2) 
%
%
%% \cite{NasuConstant} should also have the recoding -- but it is all in 
%% a language of graph coding properties which would involve a 
%% deconstruction of notation.   

\begin{theorem} \label{thm:pettype}
A subshift flow equivalent to an  irreducible PET sofic 
shift must also be irreducible PET sofic.
 
Suppose $X$ and $X'$ are irreducible PET 
sofic 
shifts with 
Fisher covers $\pi\colon  X_A \to X$ and $\pi' \colon  X_{A'}\to X'$.  
Then the following are equivalent. 
\begin{enumerate} 
\item 
The shifts $X$ and $X'$ are flow equivalent. 
\item 
The SFTs $X_A,X_{A'}$ are flow equivalent; 
$$\mathsf{MultiCard}(\pi )=\mathsf{MultiCard}(\pi ');$$    
and for each $k$,  
 the left full extensions of $\pi |_{\mulset_k (\pi )}$ and 
$\pi' |_{\mulset_k (\pi' )}$ are 
$\symk$-flow equivalent left $\symk$-SFTs. 
\end{enumerate} 
\end{theorem} 

\begin{proof} 
The invariance of the irreducible sofic PET class under 
flow equivalence is clear. 
It follows from Proposition \ref{feequiv} 
that condition $(2)$ is equivalent to 
the following condition  $(2')$: 
the SFTs $X_A,X_{A'}$ are flow equivalent, and  
for each $k$, 
 the factor maps $\pi |_{\mulset_k (\pi )}$ and 
$\pi'|_{\mulset_k (\pi' )}$ are 
flow equivalent.  By the Reduction Theorem 
\ref{theorem:reduction}, we have 
$(1)\iff (2')$. Theorem \ref{thm:pettype} follows. 
\end{proof}

{We next point out a specific case of
  Theorem \ref{thm:pettype}. 
\begin{theorem} \label{aftpetcor} 
  Suppose $Y$ is AFT with Fischer cover $\pi \colon  X\to Y$
  for which $\mathsf{MultiCard}(\pi )$ is a singleton $\{k\}$
  (e.g., if $Y$ is AFT and no point of $Y$ has more than 2
  preimages under $\pi$). 
    Then $Y$ is PET and $Y$ is classified up to
  flow equivalence by the invariants of Theorem
  \ref{thm:pettype}.
\end{theorem}
\begin{proof}
  Because $Y$ is AFT, $\mulset_k (\pi )=\mulset (\pi )$
  is closed, so $Y$ is PET and is therefore classified up to
  flow equivalence by the invariants of Theorem
  \ref{thm:pettype}.
\end{proof}
}
\begin{remark} 
Assuming the AFT shifts $X$ and $X'$ are strictly sofic, 
the SFTs $X_A,X_{A'}$ in the statement of Theorem 
\ref{thm:pettype} must be nontrivial 
irreducible SFTs (as a strictly sofic irreducible AFT 
shift has more than one orbit); they are then flow equivalent  
if and only if 
$\cok (I-A)\cong \cok (I-A')$ and  
$\det (I-A) = \det (I-A')$ (see \cite{jf:fesft}). 
For any finite group $G$, 
complete algebraic invariants for $G$-SFTs
are known (by \cite{BSullivan} for mixing extensions 
and by \GFE\ in general), but are considerably more complicated.
%for these, see \GFE.  
\end{remark} 

\begin{remark} \label{fullforce} 
Suppose $Y,Y'$ are SFTs. Let $X$ be a mixing SFT 
containing disjoint copies $Y_1,Y_2$ of $Y$ and also 
containing disjoint copies $Y'_1, Y'_2$ of $Y'$. 
Let $\alpha\colon Y_1\to Y_2$ be a topological conjugacy. 
Define a sofic shift $T$ as the quotient 
 $\pi\colon X\to T$ where $\pi (x) = \pi (\alpha (x))$ if $x\in Y_!$ and 
$\pi$ identifies no other points. Similarly define $T'$ via a
conjugacy 
$\alpha'\colon Y'_1\to Y'_2$. By Theorem \ref{thm:pettype}, 
$T$ is flow equivalent to $T'$ if and only if $Y$ is flow equivalent 
to $Y'$. Thus the classification of irreducible sofic shifts up to
flow 
equivalence requires the full classification of general (reducible)
SFTs up to flow equivalence. (Indeed, this was one motivation for
Huang's original investigation \cite{dh:fersft}.) There is a decision
procedure 
for determining whether two SFTs are flow equivalent 
\cite{bs:decide}.  
\end{remark}

%\annotation{Is Remark \ref{fullforce} redundant?}

\begin{remark} \label{rem:reducetomixingcase}
If for example $\pi$ is an $N$-point extension between mixing 
SFTs, then the full extension will be a disjoint union of 
topologically conjugate irreducible 
SFTs \cite[p.\ 495]{akmfactor}. These are not necessarily $G$-invariant;  
still, in this case 
their equivariant flow equivalence classification can be 
quickly reduced to the 
flow equivalence classification of $G$-SFTs which are mixing (see \cite[Section 4]{BSullivan} or \GFEx{Section 3}). 

But for the
classification of general PET irreducible sofic shifts, the 
classification of general reducible $G$-SFTs is required. 
Indeed,
suppose $G$ is a  finite group and 
two $N$-point extensions are the quotient maps of $G$-actions;  
then the  $N$-point extensions are topologically conjugate if and only if 
 the $G$-actions are topologically conjugate (there is no need to 
introduce the full extension). 
For every $G$-SFT, its quotient map  
can be realized as  the restriction of a Fischer cover of an irreducible PET 
sofic shift to its multiplicity shift.  Thus the 
problem of classifying 
sofic shifts up to conjugacy (or flow equivalence) 
contains the problem of classifying $G$-SFTs.  
\end{remark}

\begin{remark}\label{petvsaft} 
For  irreducible sofic shifts, let's note how 
  the PET condition is a natural refinement of
  the AFT condition. Suppose $\pi \colon X_A\to Y$
  is a Fischer cover. Define (recalling Definition \ref{defn:multiplicity} and \eqref{Epikref}, \eqref{Epiref})
  \begin{align*} 
    \widetilde{\mulsetdown}_{\leq k} (\pi ) &=
    \bigcup_{j\leq k}    E_j(\pi) \\
    \mulset_{\geq k}(\pi )&= \bigcup_{j\geq k}  \mulset_j(\pi).
  \end{align*}
  Then $Y$ is AFT if and only if
  $     \widetilde{\mulsetdown}_{\leq 1} (\pi )$ is isolated
  in $E(\pi)$;
  $Y$ is PET if and only if for each $k$,
  $     \widetilde{\mulsetdown}_{\leq k} (\pi )$ is isolated in $E(\pi)$. 
  Similarly,
  $Y$ is AFT if and only if $\mulset_{\geq 2}(\pi )$ is closed;
  $Y$ is PET if and only if $\mulset_{\geq k}(\pi )$ is closed
  for each $k\geq 2$. %AGAIN NOTATION TROUBLE
  \end{remark}   

\begin{remark} \label{lastrites} Given a Fischer cover $\pi\colon  X\to Y$.
Define $\pit \colon  E(\pi)\to \mulset (\pi) $ by $(x,w)\mapsto x$. 
  Then the following are equivalent.
  \begin{enumerate}
  \item
    $Y$ is PET.
  \item
    On each indecomposable component of $E(\pi)$, $\pit$ is constant-to-one;
    and if $C,C'$ are indecomposable components  $E(\pi)$ with
    $\pit (C) \cap \pit (C') \neq \emptyset$, then
    $\pit (C) = \pit (C')$.
  \end{enumerate}
  \end{remark} 

\section{Algorithms for PET sofic shifts} \label{sec:petalg}
  
In this section,  we briefly address decision procedures, computations 
and range of invariants for irreducible PET sofic shifts. 
Throughout, $\pi\colon  X_B\to Y$ is a given right Fischer cover of 
an irreducible sofic shift $Y$. 
This cover is presented by a graph $G_B$, with  edges of $G_B$ 
labeled by elements of  $\alp{Y}$, the alphabet
of $Y$, according to the 
1-block code $\pi$. 

{\bf Deciding whether $Y$ is PET.}

\newcommand{\sel}[1]{{\mathcal T}_{#1}}
\newcommand{\seln}[1]{{\mathcal T}_{#1}^0}
\newcommand{\ttt}{\mathbf t}
\newcommand{\pp}{\mathbf p}
\newcommand{\myex}[9]{\left(\begin{smallmatrix}#1&#2&#3\\#4&#5&#6\\#7&#8&#9\end{smallmatrix}\right)}

To begin, we describe a variant of the subset construction.  
Enumerate the vertices of the graph $G_B$ given by $B$ as
$\{1,\dots,n\}$ and for $1\leq k \leq n$, 
let $\sel{k}$ denote the set of  ordered tuples consisting of 
$k$ distinct elements drawn from $\{1,\dots ,k\}$, 
with elements written in increasing order.  We
denote an element $i$ of $\sel k$ 
by $\fbox{$i_1i_2\cdots i_k$}\, $, with 
$i_1<i_2<\cdots <i_k$, and 
 set $\sel{} = \bigcup_{1\leq k \leq n} \sel{k}$ .  
For $i = \fbox{$i_1i_2\cdots i_k$}\, $, we also 
 define $|i| = k$ and  
$\{ i \} = \{i_1, \dots , i_k\}$,  and  define 
$f(i,a)$ to be the set of terminal vertices  of $G_B$ edges 
which have label $a$ and which 
have initial vertex in $\{i\}$. 

%With $\Phi$ the block code inducing $\pi$ we think of $G_B$ as being labeled by $\Phi$.

\newcommand{\myg}[2]{\mathcal G[#1]_{#2}}
\newcommand{\mygn}[2]{\widetilde{\mathcal G[#1]}_{#2}}

From the labeled graph $G_B$ we will recursively construct
 vertex  sets 
$\mathcal V (m)$ with the aim of defining an $\alp{Y}$-labeled graph. $\mathcal V (0)$  is the singleton 
$\{i\}$ such that $|i|=n$, and
given $ \mathcal V(m)$, we define 
\begin{gather*} 
  \mathcal V(m+1) =\\ \mathcal V(m)
  \cup \left\{j \in \sel \colon  \exists a\in \alp{Y},i\in\mathcal V(m)
  \text{ such that }  
f(i,a)  = \{j\} \right\}. 
\end{gather*} 
Take $M$ such that  $\mathcal V(M) = \mathcal V(M+1)$.
We equip $\mathcal V(M)$ with an edge set $\mathcal E(M)$ as follows.
For $i= \fbox{$i_1\cdots i_k$}$ and
$j = \fbox{$j_1\cdots j_{\ell}$}$   
 in $\mathcal V(M)$, there is an edge from $i$ to $j$ labeled 
 $a$ if and only if the following hold:
 \begin{enumerate}
 \item
   $\{ j\}=f(i,a)$ , and 
 \item
   if $k> \ell =1$, then at least two edges with initial vertex
   in $\{ i \}$ have label $a$.
   \end{enumerate} 
 
Now let $\gb (\pi )$  be the maximum labeled 
subgraph of $G(M)=(\mathcal V(M),\mathcal E(M))$ such that every vertex has an incoming edge 
and an outgoing edge. Let $\vb_k$ be the set of vertices $i$ of 
$\gb (\pi )$ such that $|i|=k$. Let $\gb_k(\pi )$ be the labeled subgraph 
of $\gb (\pi )$ with vertex set $\vb_k$.   
Note, because $\pi$ is right resolving, if there is an edge in 
$\gb (\pi )$ from $i$ to $j$, then $|i|\geq |j|$.

Let $\Bb$ be an adjacency matrix for the graph $\gb (\pi )$.
Let $\Bb_k$ be the principal submatrix of $\Bb$
on its indices in $\Vb_k$.
Let $\Xb$ denote $X_{\Bb}$ and let $\Xb_k$ be the subshift 
of $\Xb$ which is the edge shift defined from
the subgraph of  $\gb (\pi )$ 
with vertex set $\Vb_k$.
Let $\overline{\phi} \colon  X_{\Bb}\to Y$ be the one-block map  given 
from the edge labeling of $\gb (\pi )$,
and let $\overline{\phi}_k$ be the restriction of
$\overline{\phi}$ to $\Xb_k$. 
Then $\overline{\phi}_k$ maps $\Xb_k$ onto the subset of $\mulsetdown_k(\pi )$ 
whose $k$ preimages are uniformly separated.

Suppose $\gb (\pi )= \gb$ contains an edge $i\to j$ for
which $|i|> |j|$.  Then there exist
$k> h$ and a biinfinite 
path $x$ in $\gb $ such that $x_n$ is a  $\gb_k$ 
edge for all but finitely many negative  $n$ 
and $x_n$ is a  $\gb_h$ 
edge for all but finitely many positive $n$.
If $h=1$, then using condition (2) we conclude $Y$ is not AFT,
hence not PET. 
If $h>1$ and $Y$ is AFT, then $\pi$ is biclosing and
we have points of $\mulset_k(\pi )$ in the closure of $\mulset_h(\pi )$,
and $\pi$ is not PET. 

We summarize with the following proposition.

\begin{proposition}
  For a Fischer cover $\pi\colon  X\to Y$ for an irreducible $Y$,
the following are equivalent. 
\begin{enumerate} 
\item 
$Y$ is PET. 
\item 
$\gb (\pi  ) = \bigcup_{k} \gb_k\,  (\pi )$ \\ 
(i.e., if $e$ is an edge in $\gb (\pi  )$ from $i$ to $j$, 
then $|i|=|j|$).  
\end{enumerate} 
\end{proposition}

\begin{examples} \label{petalgexamples}
  Below are 
  two matrices which present labeled graphs of
  Fischer covers, $\pi\colon  X\to Y$.    
  \[
  B=\begin{pmatrix} a+f&0&c\\0&a&b\\d&b&a\end{pmatrix}, \qquad 
  C=\begin{pmatrix} d&e&0&0\\f&a+d&b&c\\0&c&a&b\\0&b&c&a\end{pmatrix}
  . 
  \]
  (For example $B$ presents a labeling of a graph with adjacency matrix
  $\left(\begin{smallmatrix}2&0&1\\0&1&1\\1&1&1\end{smallmatrix}\right)$.)
  For $\pi$ defined by $B$, the sofic shift $Y$ is AFT but not PET;
  the subshift $\mulset_2 (\pi )$ is not closed, and
  in $\gb (\pi )$ 
we see the edge 
\begin{equation}
\xymatrix{
{\fbox{123}}\ar[rrr]^-{b}&&&{\fbox{23}}}.
\end{equation}
For $Y$ defined by $C$, $Y$ is PET;
the vertices of $\gb(\pi )$ are 
${\fbox{234}}$ and  ${\fbox{12}}$, and there
is no edge between these.
(Here,
${\fbox{234}}$ and  ${\fbox{12}}$ have in common the
vertex entry $2$. In general, for $k>\ell \geq 2$,
if $i\in \mathcal V_k$ and $j\in \mathcal V_{\ell}$
and $i,j$ have more than one vertex entry in common,
then $Y$ is not PET.) 
\end{examples}

\begin{remark}
Note that thus far we might as well have indexed the vertices of $\gb(\pi)$ by sets rather than ordered tuples, as in 
\cite{MR2923456}.  As we will see below, the ordering becomes necessary for reading off the skewing functions.
\end{remark}
{\bf  Computing $\Z_+G$ matrices for  PET systems.} 

Suppose that $Y$ is PET, with right Fischer cover
$\pi \colon  X_B \to Y$ and with $\pib \colon  \Xb \to Y$ constructed
as above. Suppose $k\in \mathsf{MultiCard}(\pi )$.
%Let $\pib \colon  \Xb_k \to Y$ is a topological conjugacy onto
%its image, which is $\mulset_k(\pi )$.

We define a labeling of the edges of the
graph $\gb_k$ as follows. Given such an edge $E$, 
from vertex $i$ to vertex $j$ in $\vb_k$,
we let $ \tau (E)$  be
the unique permutation $\tau$ 
of $\{ 1, \dots , k\} $ such that
for $1\leq t\leq k$ there is an edge labeled $a$
from $i_t$ to $j_{\tau (t)}$. 
%, and label the edge $E$ 
%with $\tau (E) $.
The skewing function defined by $x\mapsto \tau (x_0)$
defines a $k$ point extension
$(\Xb_k\times \{1, \dots , k\}, \overline{\rho}_k)$
of $(\Xb_k, \sigma )$.
The square
matrix $B_k$ over $\Z_+\symk $  which
presents this extension is defined as follows:
$B_k (i,j) = \sum_E \tau (E)$, where the sum is over the 
edges $E$ from $i$ to $j$ in $\gb_k$. The matrix
$\Bb_k$ defined earlier is indeed the image of $B_k$ under the
entrywise augmentation map, $\sum_{g\in \symk}  n_gg \mapsto \sum_g n_g$. 

Let $\{\tau_j\colon  1\leq j \leq k\} $ be the
collection of continuous sections
such that $\mulset (\pi )$ is the disjoint union of
the $k$ sets $\tau_j(\mulsetdown_k(\pi ))$.
Explicitly, a point $y$ in $\mulsetdown_k (\pi)$
has $k$ preimages, whose zero coordinate symbols $x_0$ are
edges with initial vertices $i_1< i_2 < \cdots < i_k$;
we choose $\tau_j (y)$ to be the $x$ such that $x_0$
has initial vertex $i_j$. Then 
there is a commuting diagram 

\[
\xymatrix{
{\Xb_k\times \{1, \dots , k\}}
\ar[r]^-{\phi_k}\ar[d]_-{p_k}& 
{\mulset_k (\pi )}\ar[d]_-{\pi_k}    \\
{\Xb_k}\ar[r]_-{\pib_k}&{\mulsetdown_k (\pi )} 
}
\] 
in which $p_k\colon  (x,j) \mapsto x$; $\overline{\phi}$ is
defined by the edge labeling of $\gb_k$; and
$\phi\colon  (x,j) \mapsto \tau_j (\overline{\phi}_k (x))$.
The maps $\phi_k $ and $\overline{\phi}_k$ are homeomorphisms, and
$\pi_k\phi_k = \overline{\phi}_kp_k$. Thus the $k$ point extensions given
by $p_k$ and $\pi_k$ are topologically conjugate.

By Proposition \ref{skewconjugateleftaction},  
the SSE-$\Z_+S_k$ class of the matrix $B_k$ is a complete
invariant for the conjugacy class of the extension $\pi_k$.
The list of SSE-$\Z_+S_k$ classes ($k\in \mathsf{MultiCard}(\pi )$) is therefore a complete invariant for the restriction of
$\pi$ to $M^{-1}(\pi)$. By the extension result \cite[Theorem 1.5]{mbwk:ass}, this list, together with the
SSE-$\Z_+$ class of the matrix $B$ defining the mixing SFT
which is the domain of
$\pi$, determines the conjugacy class of $\pi$
up to finitely many possibilities.

By Theorem \ref{thm:pettype}, the classification
of PET sofic shifts up to flow equivalence is reduced to
the (known) classification of irreducible SFTs up to flow equivalence
and the classification of $G$-SFTs up to flow equivalence.
Complete algebraic invariants for this are given in 
 \GFE. 

{\bf Range of invariants.} 

First, suppose $X$ is a mixing SFT, $\mathcal K$ is a
finite set of positive integers from $[2,\infty )$
  and for $k\in \mathcal K$ we are given a $k$-point
  extension of SFTs, $\psi_k \colon  U_k\to V_k$.
  Let $\psi$ be the disjoint union of the $\psi_k$,
  with domain the disjoint union $U$ of the $U_k$. 
  
Then the following are equivalent.
\begin{enumerate}

\item
  There is a PET shift $Y$ with Fischer cover $\pi$
  whose domain is topologically conjugate to $X$
  and whose restriction to $\mulset (\pi)$ is conjugate
  to $\psi$.
\item
The subshift  $U$ embeds into $X$.
\end{enumerate}
Necessary  and sufficient conditions for $U$ embedding into $X$ are
given by entropy and periodic point counts
according to Krieger's Embedding Theorem. So,  the embedding
constraint is the only  constraint for realizing invariants.

Now consider whether $\psi$ can be realized in a specified
flow equivalence class of nontrivial mixing SFTs. Trivially,
it can,
because every such class is contains SFTs defined by arbitrarily
large matrices with arbitrarily large entries.

%%\annotation{SE: I miss the near Markov remark}
%    %\fi
% 
%%\bibliographystyle{amsalpha}
%\bibliographystyle{plain}
%
%%\bibliography{/Users/eilers/Professional/Reference/sds,/Users/eilers/Professional/Reference/own}
%\bibliography{sds}

\begin{thebibliography}{10}

\bibitem{akmfactor}
R.~Adler, B.~Kitchens, and B.~Marcus.
\newblock Almost topological classification of finite-to-one factor maps
  between shifts of finite type.
\newblock {\em Ergodic Theory Dynam. Systems}, 5(4):485--500, 1985.

\bibitem{akmgroup}
R.~L. Adler, B.~Kitchens, and B.~H. Marcus.
\newblock Finite group actions on shifts of finite type.
\newblock {\em Ergodic Theory Dynam. Systems}, 5(1):1--25, 1985.

\bibitem{mbbd:citosts}
M.~Barge and B.~Diamond.
\newblock A complete invariant for the topology of one-dimensional substitution
  tiling spaces.
\newblock {\em Ergodic Theory Dynam. Systems}, 21(5):1333--1358, 2001.

\bibitem{MR2923456}
J.~Berstel, C.~De~Felice, D.~Perrin, C.~Reutenauer, and G.~Rindone.
\newblock Recent results on syntactic groups of prefix codes.
\newblock {\em European J. Combin.}, 33(7):1386--1401, 2012.

\bibitem{BowenFranks}
R.~Bowen and J.~Franks.
\newblock Homology for zero-dimensional nonwandering sets.
\newblock {\em Ann. of Math. (2)}, 106(1):73--92, 1977.

\bibitem{B-LowerEntropyFactors}
M.~Boyle.
\newblock Lower entropy factors of sofic systems.
\newblock {\em Ergodic Theory Dynam. Systems}, 3(4):541--557, 1983.

\bibitem{mb:fesftpf}
M.~Boyle.
\newblock Flow equivalence of shifts of finite type via positive
  factorizations.
\newblock {\em Pacific J. Math.}, 204(2):273--317, 2002.

\bibitem{bce:fei}
M.~Boyle, T.~Carlsen, and S.~Eilers.
\newblock Flow equivalence and isotopy for subshifts.
\newblock arXiv:1511.03478, November 2015.

\bibitem{bce:gfe}
M.~Boyle, T.M. Carlsen, and S.~Eilers.
\newblock Flow equivalence of {G-SFTs}.
\newblock arXiv preprint, November 2015.

\bibitem{BHspectra}
M.~Boyle and D.~Handelman.
\newblock The spectra of nonnegative matrices via symbolic dynamics.
\newblock {\em Ann. of Math. (2)}, 133(2):249--316, 1991.

\bibitem{mbdh:pbeim}
M.~Boyle and D.~Huang.
\newblock Poset block equivalence of integral matrices.
\newblock {\em Trans. Amer. Math. Soc.}, 355(10):3861--3886 (electronic), 2003.

\bibitem{bkm}
M.~Boyle, B.~Kitchens, and B.~Marcus.
\newblock A note on minimal covers for sofic systems.
\newblock {\em Proc. Amer. Math. Soc.}, 95(3):403--411, 1985.

\bibitem{mbwk:amsess}
M.~Boyle and W.~Krieger.
\newblock Almost {M}arkov and shift equivalent sofic systems.
\newblock In {\em Dynamical systems (College Park, MD, 1986--87)}, volume 1342
  of {\em Lecture Notes in Math.}, pages 33--93. Springer, Berlin, 1988.

\bibitem{mbwk:ass}
M.~Boyle and W.~Krieger.
\newblock Automorphisms and subsystems of the shift.
\newblock {\em J. Reine Angew. Math.}, 437:13--28, 1993.

\bibitem{mbdldr:agsft}
M.~Boyle, D.~Lind, and D.~Rudolph.
\newblock The automorphism group of a shift of finite type.
\newblock {\em Trans. Amer. Math. Soc.}, 306(1):71--114, 1988.

\bibitem{BoSc2}
M.~Boyle and S.~Schmieding.
\newblock Finite group extensions of shifts of finite type: {$K$}-theory,
  {P}arry and {L}iv{\v{s}}ic.
\newblock arXiv:1503.02050, to appear in \emph{Ergodic Theory Dynam. Systems},
  2015.

\bibitem{bs:decide}
M.~Boyle and B.~Steinberg.
\newblock Decidability of matrix equivalence and diagram isomorphism, for
  {Cuntz-Krieger} {$C^*$}-algebras and flow equivalence (working title).
\newblock In preparation, 2015.

\bibitem{BSullivan}
M.~Boyle and M.C. Sullivan.
\newblock Equivariant flow equivalence for shifts of finite type, by matrix
  equivalence over group rings.
\newblock {\em Proc. London Math. Soc. (3)}, 91(1):184--214, 2005.

\bibitem{costasteinberg}
A.~Costa and B.~Steinberg.
\newblock Flow equivalence of {$R$}-graph shifts.
\newblock arXiv:1304.3487; to appear in \emph{Ergodic Theory Dynam. Systems},
  2014.

\bibitem{jf:fesft}
J.~Franks.
\newblock Flow equivalence of subshifts of finite type.
\newblock {\em Ergodic Theory Dynam. Systems}, 4(1):53--66, 1984.

\bibitem{FujiwaraOsikawa}
M.~Fujiwara and M.~Osikawa.
\newblock Sofic systems and flow equivalence.
\newblock {\em Math. Rep. Kyushu Univ.}, 16(1):17--27, 1987.

\bibitem{HamachiNasu}
T.~Hamachi and M.~Nasu.
\newblock Topological conjugacy for {$1$}-block factor maps of subshifts and
  sofic covers.
\newblock In {\em Dynamical systems ({C}ollege {P}ark, {MD}, 1986--87)}, volume
  1342 of {\em Lecture Notes in Math.}, pages 251--260. Springer, Berlin, 1988.

\bibitem{dh:fersft}
D.~Huang.
\newblock Flow equivalence of reducible shifts of finite type.
\newblock {\em Ergodic Theory Dynam. Systems}, 14(4):695--720, 1994.

\bibitem{rj:xxx}
R.~Johansen.
\newblock Flow equivalence of sofic beta-shifts.
\newblock To appear, \emph{Ergodic Theory Dynam. Systems}.

\bibitem{JungOpen}
U.~Jung.
\newblock Open maps between shift spaces.
\newblock {\em Ergodic Theory Dynam. Systems}, 29(4):1257--1272, 2009.

\bibitem{hk:book}
A.~Katok and B.~Hasselblatt.
\newblock {\em Introduction to the modern theory of dynamical systems},
  volume~54 of {\em Encyclopedia of Mathematics and its Applications}.
\newblock Cambridge University Press, Cambridge, 1995.
\newblock With a supplementary chapter by Katok and Leonardo Mendoza.

\bibitem{KimRoushsoficse}
K.~H. Kim and F.~W. Roush.
\newblock An algorithm for sofic shift equivalence.
\newblock {\em Ergodic Theory Dynam. Systems}, 10(2):381--393, 1990.

\bibitem{bpk:sd}
B.P. Kitchens.
\newblock {\em Symbolic dynamics; {O}ne-sided, two-sided and countable state
  {M}arkov shifts}.
\newblock Springer-Verlag, Berlin, 1998.

\bibitem{KriegerSoficI}
W.~Krieger.
\newblock On sofic systems. {I}.
\newblock {\em Israel J. Math.}, 48(4):305--330, 1984.

\bibitem{kriegerfe}
W.~Krieger.
\newblock On flow equivalence of {$R$}-graph shifts.
\newblock arXiv:1406.6283, 2014.

\bibitem{dlbm:isdc}
D.~Lind and B.~Marcus.
\newblock {\em An introduction to symbolic dynamics and coding}.
\newblock Cambridge University Press, Cambridge, 1995.

\bibitem{bm:ssed}
B.~Marcus.
\newblock Sofic systems and encoding data.
\newblock {\em IEEE Trans. Inform. Theory}, 31(3):366--377, 1985.

\bibitem{Matsumoto2013}
K.~Matsumoto.
\newblock {$C^*$}-algebras associated with lambda-synchronizing subshifts and
  flow equivalence.
\newblock {\em J. Aust. Math. Soc.}, 95(2):241--265, 2013.

\bibitem{NasuConstant}
M.~Nasu.
\newblock Constant-to-one and onto global maps of homomorphisms between
  strongly connected graphs.
\newblock {\em Ergodic Theory Dynam. Systems}, 3(3):387--413, 1983.

\bibitem{Nasumask}
M.~Nasu.
\newblock Topological conjugacy for sofic systems and extensions of
  automorphisms of finite subsystems of topological {M}arkov shifts.
\newblock In {\em Dynamical systems ({C}ollege {P}ark, {MD}, 1986--87)}, volume
  1342 of {\em Lecture Notes in Math.}, pages 564--607. Springer, Berlin, 1988.

\bibitem{ns:book}
V.~V. Nemytskii and V.~V. Stepanov.
\newblock {\em Qualitative theory of differential equations}.
\newblock Princeton Mathematical Series, No. 22. Princeton University Press,
  Princeton, N.J., 1960.

\bibitem{parrysullivan}
W.~Parry and D.~Sullivan.
\newblock A topological invariant of flows on {$1$}-dimensional spaces.
\newblock {\em Topology}, 14(4):297--299, 1975.

\bibitem{restorff}
G.~Restorff.
\newblock Classification of {C}untz-{K}rieger algebras up to stable
  isomorphism.
\newblock {\em J. Reine Angew. Math.}, 598:185--210, 2006.

\bibitem{rordam}
M.~R{\o}rdam.
\newblock Classification of {C}untz-{K}rieger algebras.
\newblock {\em $K$-Theory}, 9(1):31--58, 1995.

\bibitem{rudolphcounts}
D.J. Rudolph.
\newblock Counting the relatively finite factors of a {B}ernoulli shift.
\newblock {\em Israel J. Math.}, 30(3):255--263, 1978.

\bibitem{WilliamsMinimal}
S.~Williams.
\newblock Covers of non-almost-finite type sofic systems.
\newblock {\em Proc. Amer. Math. Soc.}, 104(1):245--252, 1988.

\end{thebibliography}

\mbox{}\\ \small
\noindent
\textsc{Department of Mathematics, 
University of Maryland,
College Park, MD 20742-4015, USA}

\emph{E-mail address: }\texttt{mmb@math.umd.edu}\\[0.5cm]

\noindent
\textsc{Department of Science and Technology, University of the Faroe Islands, N\'oat\'un 3, FO-100 T\'orshavn, The Faroe Islands}

\emph{E-mail address: }\texttt{toke.carlsen@gmail.com}\\[0.5cm]

\noindent
\textsc{Department of Mathematical Sciences,  University of Copenhagen, DK-2100 Copenhagen \O, Denmark}

\emph{E-mail address: }\texttt{eilers@math.ku.dk}\\

\end{document}